\documentclass[10pt]{article}
\usepackage[margin=1.15in]{geometry} 
\usepackage{amsmath,amsthm,amssymb}
\usepackage{float}
 \usepackage{hyperref}
 \usepackage{authblk}

\newcommand{\E}{\mathbb{E}}
\newcommand{\ignore}[1]{}
\usepackage[dvipsnames]{xcolor}
\usepackage{multirow}
\usepackage{stmaryrd}
\usepackage{bm}
\usepackage{caption}
\usepackage{cleveref}
\usepackage{tcolorbox}
\usepackage{MnSymbol}
\newtheorem{theorem}{Theorem}
\newtheorem{remark}{Remark}

\usepackage{mathtools}
\newtheorem{lemma}{Lemma}

\newtheorem{definition}{Definition}

\usepackage{algorithm}
\usepackage{algpseudocode}
\usepackage{subcaption}
\usepackage{cleveref}

\numberwithin{equation}{section}

\graphicspath{{..}}

\newcommand{\an}[1]{{\leavevmode\color{red}{#1}}}
    \makeatletter
\def\@fnsymbol#1{\ensuremath{\ifcase#1\or \dagger\or \ddagger \or 
   \mathsection\or \mathparagraph\or \|\or **\or \dagger\dagger
   \or \ddagger\ddagger \else\@ctrerr\fi}}
    \makeatother

\begin{document}
 
 
\title{Structure-Preserving Discontinuous Galerkin Methods for Stochastic Shallow Water Equations}

\author[1]{Yekaterina Epshteyn \thanks{\href{epshteyn@math.utah.edu}{epshteyn@math.utah.edu}}}

\author[1,2]{Akil Narayan \thanks{\href{akil@sci.utah.edu}{akil@sci.utah.edu}}}

\author[1]{Yinqian Yu \thanks{ \href{yinqian.yu@utah.edu}{yinqian.yu@utah.edu}}}

\affil[1]{\small Department of Mathematics, University of Utah, 155 S 1400 E, Salt Lake City, 84112, UT, USA}
\affil[2]{\small Scientific Computing and Imaging Institute, University of Utah, 72 S Central Campus Dr, Salt Lake City, 84112, UT, USA}

\date{ }
\maketitle
\begin{abstract}
  Shallow water equations (SWE) are fundamental models in fluid dynamics that are
essential for studying a wide range of geophysical and engineering phenomena. In many practical applications, uncertainties arising from initial conditions and bottom topography must be taken into account, motivating the development of stable and accurate numerical methods for stochastic SWE.
Building on the hyperbolicity-preserving stochastic Galerkin formulation for SWE
[Dai, Epshteyn, Narayan, 2021 SISC] and a stochastic extension of the entropy stable discontinuous Galerkin methods for skew-symmetric SWE [Fu, 2022 JSC], we develop a structure-preserving, entropy conservative, and entropy stable discontinuous Galerkin--stochastic Galerkin method for the stochastic shallow water system, with the well-balanced property. We demonstrate the accuracy, applicability, and robustness of the proposed structure-preserving algorithms through
several numerical experiments.

\end{abstract}



\section{Introduction}\label{intro}
The Saint-Venant system of shallow water equations (SWE) is a classical and widely used model in the mathematical and physical description of geophysical fluid flows, particularly in regimes where the horizontal length scale is significantly larger than the vertical one, such as rivers, lakes, and coastal regions \cite{de1871theorie}. This system is capable of capturing the essential dynamics of wave propagation, transport processes, and general hydrodynamic behavior under the shallow-water assumption, making it a fundamental model for a broad range of practical applications involving water flows. Despite its wide applicability, deterministic shallow water models are often insufficient in realistic settings, since environmental data and operating conditions are frequently uncertain or partially unknown. Incorporating such uncertainty is therefore essential for improving predictive capability, which motivates the development of \textit{parameterized} stochastic shallow water equations (SWE), in which randomness is introduced through random inputs to model uncertainty propagation.

To represent uncertainty, we employ the polynomial chaos expansion (PCE) framework, which provides a systematic approach for representing random fields in terms of orthogonal polynomial bases. Originally introduced by Wiener \cite{wiener1938homogeneous} for Gaussian processes using Hermite polynomials, the method has been extended to more general orthogonal polynomial families suitable for arbitrary probability distributions \cite{ghanem2003stochastic,le2004uncertainty,wan2005adaptive,xiu2002wiener}. In this framework, uncertainty is encoded through a truncated spectral expansion, which transforms the stochastic partial differential equation into a coupled deterministic system governing the expansion coefficients. Two main classes of numerical methods arise from this formulation: intrusive and non-intrusive approaches.

Non-intrusive methods, based on sampling strategies \cite{mishra2012multilevel,nobile2008sparse,xiu2005high}, construct the stochastic solution by repeatedly solving deterministic realizations of the system at selected samples of the random variables. This approach has the advantage of directly utilizing existing and well-established deterministic solvers for the shallow water equations, including a wide range of finite volume and related methods \cite{zhou2001surface,kurganov2002central,rogers2003mathematical,vcrnjaric2004balanced,xing2005high,xing2006high,xing2006new,kurganov2007second,BEKP,epshteyn2023adaptive,kurganov2018finite,liu2018well,xing2016high,XING2017361,zhong2022entropy}. However, while convenient and flexible, non-intrusive approaches may suffer from reduced accuracy compared to spectral methods, and it is often difficult to guarantee important physical properties such as entropy stability, and the ability of preserving equilibrium states. In contrast, intrusive stochastic Galerkin (SG) methods project the truncated PCE representation onto the stochastic basis, leading to a coupled deterministic system for the expansion coefficients \cite{wiener1938homogeneous,xiu2002wiener}. Although this approach requires significant modifications to existing deterministic solvers and results in increased algebraic and computational complexity, it is generally expected to achieve higher-order accuracy than non-intrusive methods. Since the SG method is based on projection, it leads to near-optimal accuracy in the $L^2$ sense for static problems \cite{babuska2004galerkin,le2010spectral}. Moreover, the SG framework facilitates rigorous analysis of both the resulting SG system and its associated numerical discretizations. As a consequence, SG methods have been successfully applied to uncertainty quantification in diffusion problems \cite{xiu2009efficient,eigel2014adaptive}, kinetic equations \cite{hu2016stochastic,shu2017stochastic}, and conservation and balance laws with symmetric Jacobian matrices \cite{tryoen2010intrusive}. Nevertheless, for general nonlinear hyperbolic systems such as the shallow water equations and the Euler equations, the resulting SG system may lose hyperbolicity \cite{despres2013robust,gerster2019hyperbolic,jin2019study}, leading to ill-posedness, potentially producing unphysical solutions, and compromising the robustness of the subsequent numerical schemes.

To address these challenges, various structure-preserving strategies have been proposed in the literature. Recent advances include SG-based analysis and algorithms for scalar conservation laws \cite{zhou2012galerkin}, including well-balanced schemes \cite{jin2016well}, Haar wavelet-based approaches \cite{gerster2022haar}, hyperbolicity-preserving formulations based on non-equivalent Roe variables \cite{gerster2020entropies,bender2024entropy,offner2026high}, filtering techniques \cite{kusch2020filtered}, limiter-based strategies \cite{schlachter2018hyperbolicity}, hyperbolicity-preserving formulations for linear problems \cite{pulch2012generalised} and approaches based on linearization \cite{wu2017stochastic}, operator splitting methods \cite{chertock2015operator,chertock2015well}, non-conservative formulations of SG SWE systems \cite{chen2023cross}, and entropy-variable representations \cite{poette2009uncertainty,poette2019}.

In this paper, we build upon the recent work on a hyperbolicity-preserving SG formulation for the stochastic SWE \cite{dai2021hyperbolicity} and combine it with a skew-symmetric formulation of SWE \cite{fu2022high} with a stochastic extension, resulting in a hyperbolicity-preserving skew-symmetric SG system for stochastic SWE. Despite the hyperbolicity-preserving property, the SG system is a nonlinear hyperbolic conservation law or balance law, and therefore it inherits the standard challenges associated with the numerical approximation of such systems. For instance, solutions may develop shock discontinuities in finite time for generic initial data, leading to non-unique weak solutions. To select the physically relevant solution, an additional entropy condition must be enforced, either explicitly or implicitly. Moreover, the strong nonlinearity of the system poses significant challenges for implicit time discretizations \cite{dafermos2016hyperbolic,leveque1992numerical,leveque2002finite}. 

Over the past decades, a variety of high-resolution numerical methods have been developed for such hyperbolic systems, most notably finite volume (FV) and discontinuous Galerkin (DG) methods \cite{leveque2002finite,cockburn2001runge}. FV schemes provide robust shock-capturing capabilities through conservative flux formulations \cite{kurganov2000new}, while DG methods can further achieve high-order accuracy with a variational framework using piecewise polynomial approximations \cite{hesthaven2008nodal}. In recent years, considerable effort has been devoted to the design of structure-preserving schemes, including entropy-conservative and entropy-stable methods that enforce consistency with the entropy inequality \cite{tadmor2003entropy,fjordholm2011well}, as well as well-balanced schemes that exactly preserve steady states at the discrete level for balance laws such as the shallow water equations \cite{bermudez1994upwind,xing2005high,castro2010some}. These challenges inherited from nonlinear hyperbolic conservation/balance laws are further compounded in the stochastic Galerkin setting, where the coupling between stochastic modes introduces additional complexity, making the design of robust, structure-preserving numerical schemes significantly more difficult.

\subsection{Contributions of this paper}
In this paper, we develop a discontinuous Galerkin--stochastic Galerkin (DG-SG) scheme for stochastic shallow water equations (SWE), with the desired well-balanced property, entropy stability, and an applicable positivity-preserving criterion for the fully discrete numerical scheme.
\begin{itemize}\label{paper_contribution}
  \item We derive a \textit{skew-symmetric} hyperbolicity-preserving stochastic Galerkin (SG) formulation for the stochastic SWE. This formulation builds on the SGSWE in \cite{dai2021hyperbolicity}, together with the stochastic extension of the skew-symmetric SWE in \cite{fu2022high}. The resulting skew-symmetric SGSWE is essential for constructing a DG scheme that satisfies the entropy admissibility criterion to resolve the non-uniqueness of weak solutions.
    \item Based on the skew-symmetric SGSWE, we construct high-order, well-balanced, entropy conservative, and entropy stable semi-discrete DG-SG schemes for the stochastic SWE. The designed numerical fluxes in our DG-SG scheme are stochastic extensions of those developed in \cite{fu2022high} for the deterministic SWE. We present practical fully-discrete versions of semi-discrete scheme that are implementable algorithms to simulate the system.
    \item We present several challenging numerical experiments to evaluate the performance of our schemes, including accuracy, well-balanced property, energy decay, and numerical robustness.
\end{itemize}
Compared with closely related SG methods for the stochastic SWE, the present work combines a different SG formulation with a different spatial discretization. The hyperbolicity-preserving SG formulation in \cite{dai2021hyperbolicity} leads to well-balanced finite-volume-type schemes, whereas here we develop a high-order DG discretization based on a skew-symmetric SG formulation. In contrast to entropy-stable DG approaches based on Roe-variable transformations \cite{gerster2020entropies,bender2024entropy,offner2026high}, the proposed method works directly with the hyperbolicity-preserving conserved-variable SG formulation, accommodates general orthonormal polynomial bases associated with the random input distribution, and uses a velocity projection in the DG discretization. This construction yields a DG-SG scheme that is simultaneously high-order, well-balanced, entropy conservative or entropy stable, and compatible with a practical hyperbolicity-preserving criterion.

An outline of this paper is as follows: In Section \ref{Sec2}, we review the polynomial chaos expansion (PCE) framework and develop a skew-symmetric stochastic Galerkin (SG) formulation of the shallow water equations (SWE) following \cite{dai2021hyperbolicity}. In Section \ref{Sec3}, we develop a DG-SG formulation for the skew-symmetric SG system introduced in Section \ref{Sec2}, and we construct well-balanced, entropy conservative, and entropy stable numerical fluxes for our proposed DG-SG scheme. In Section \ref{Sec4}, we provide algorithm details of our scheme for practical implementation. In Section \ref{Sec5}, we present several challenging numerical examples to demonstrate the performance of the proposed schemes and verify their theoretical properties. In Section \ref{Sec6} we summarize the main results and discuss possible directions.

\section{Skew-symmetric stochastic Galerkin formulation of shallow water equations}\label{Sec2}
In this section, we introduce a skew-symmetric stochastic Galerkin formulation of the shallow water equations (SWE), which possesses a hyperbolicity-preserving property with a reasonable condition \cite{dai2021hyperbolicity}.

We begin with the stochastic SWE:
\begin{align}\label{stochasticSWE}
    U_t(x,t,\xi) + F(U)_x &= S(U), & U &= [h,q]^\top,
\end{align}
where $F(U)$ represents the flux and $S(U)$ is the source term:
\begin{align}
    F(U) &= \begin{pmatrix}
        q \\ \frac{q^2}{h}+ \frac{gh^2}{2}
    \end{pmatrix}, & S(U) & = \begin{pmatrix}
        0 \\ -ghB_x
    \end{pmatrix},
\end{align}
Here $U = [h,q]^\top$ is the vector of conservative variables, $h = h(x,t,\xi)$ is the water height, and $q=q(x,t,\xi)$ is the discharge, and $B=B(x,\xi)$ is the time-independent bottom topography. $\xi$ is a random field, which could result from uncertainty or ignorance of the inputs, for example, bottom topography and initial data. Therefore, our solution depends not only on the spatial and temporal variable $(x,t)$, but also on the stochastic random variable $\xi$. Our goal is to develop accurate, efficient, and robust numerical algorithms for the stochastic SWE \eqref{stochasticSWE}. We begin with some preliminaries for the stochastic model formulation.

\subsection{Polynomial chaos expansion}\label{Sec2_1}

In this subsection, we provide a brief review of polynomial chaos expansion (PCE). For further details, see \cite{debusschere2004numerical, sullivan2015introduction, xiu2010numerical}.
Let $\xi \in \mathbb{R}^d$ be a $d$-dimensional random variable with a Lebesgue density function $\rho$. Define the $L^2$-integrable function space associated with $\rho$ as follows:
\begin{equation}
    L_{\rho}^2 (\mathbb{R}^d) \coloneqq \left\{ f: \mathbb{R}^d \to \mathbb{R} \;\;\bigg|\;\; \left( \int_{\mathbb{R}^d} f^2(s)\rho(s) ds \right)^{1/2} < +\infty \right\}.
\end{equation}
Assuming that $\rho$ has finite polynomial moments of all orders, there exists a $d$-variate orthonormal polynomial basis $\{ \phi_k\}_{k=1}^{\infty}$ such that,
\begin{equation}
  \E [ \phi_k(\xi) \phi_l(\xi)] = \langle \phi_k, \phi_l \rangle_{\rho} \coloneqq \int_{\mathbb{R}^d} \phi_k(s)\phi_l(s)\rho(s)ds = \delta_{k,l}, \quad \forall k,l \in \mathbb{N}, \quad \phi_1(\xi) \equiv 1, 
\end{equation}
where $\delta_{k,l}$ is the Kronecker delta. Further, under mild conditions \cite{ernst2012convergence}, then these basis functions span $L^2_{\rho}$: For any $z \in L_{\rho}^2$, then
\begin{align}
  z(x,t,\xi) &\overset{L_{\rho}^2}{=} \sum_{k=1}^{\infty} \widehat{z}_k(x,t)\phi_k(\xi), & 
  \widehat{z}_k &= \left\langle z, \phi_k \right\rangle_{\rho},
\end{align}
where $x$, $t$ are deterministic spatial and temporal variables, and $\widehat{z}_k(x,t)$ are the deterministic Fourier-type coefficients corresponding to the orthonormal basis $\{ \phi_k\}_{k=1}^{\infty}$. Numerical computations require finite truncations of these expansions. Let $P = \text{span}\{ \phi_k, k=1,2,\dots,K\} $ be a $K$-dimensional polynomial subspace of $ L_{\rho}^2$. We then define the $K$-term PCE approximation of a random field $z$ on this subspace as follows:
\begin{equation}\label{PCE-truncated}
    \Pi_P[z](x,t,\xi) \coloneqq \sum_{k=1}^{K} \widehat{z}_k(x,t)\phi_k(\xi).
\end{equation}
The statistics of $\Pi_P[z]$ can be derived from its expansion coefficients. Specifically, the mean and variance of the random field $z$ can be expressed in terms of these coefficients as follows:
\begin{align}
  \mathbb{E}[\Pi_P[z](x,t,\xi)  ] &= \widehat{z}_1(x,t), & \text{Var}[\Pi_P[z](x,t,\xi) ] &= \sum_{k=2}^K \widehat{z}_k^2(x,t).
\end{align}
Our numerical schemes involve specific manipulations of truncated expansion coefficients. Let $\widehat{z} = (\widehat{z}_1,\dots,\widehat{z}_K)^\top \in \mathbb{R}^K$ be the vector of truncated $K-$term PCE coefficients of $z$, then define the linear operator $\mathcal{P}:\mathbb{R}^K \to \mathbb{R}^{K\times K}$ as follows:
\begin{align}
  \mathcal{P}(\widehat{z}) &\coloneqq \sum_{k=1}^K \widehat{z}_k \mathcal{M}_k, & 
  \mathcal{M}_k &\in \mathbb{R}^{K\times K}, & 
  (\mathcal{M}_k)_{l,m} &\coloneqq \langle \phi_k,\phi_l\phi_m  \rangle_{\rho}.
\end{align}
This operator $\mathcal{P}(\cdot)$ satisfies the symmetry and commutativity properties as follows:
\begin{align}\label{commutative_prop}
  \mathcal{P}(\widehat{z}) &= (\mathcal{M}_1 \widehat{z},\mathcal{M}_2 \widehat{z},\dots, \mathcal{M}_K \widehat{z} ), & 
  \mathcal{P}(\widehat{a})\widehat{b} &= \mathcal{P}(\widehat{b})\widehat{a}, &
  \widehat{b}^\top \mathcal{P}(\widehat{a}) &= \widehat{a}^\top \mathcal{P}(\widehat{b}),
\end{align}
The last two properties are proved in \cite[Lemma 2.1]{dai2024energy}, using the definition and symmetry of $\mathcal{P}$.
A stochastic Galerkin (SG) formulation of a $\xi$-parameterized partial differential equation (PDE) assumes that the state variable lies in the polynomial space $P$ and forms a scheme corresponding to projecting the PDE residual onto $P$. Note that the hyperbolicity of the straightforward SG formulation for nonlinear hyperbolic PDEs, such as the shallow water equations, is not automatically guaranteed. Therefore, special designs are required for the SG formulation of these nonlinear hyperbolic PDEs to preserve such an important property.

\subsection{Hyperbolicity-preserving stochastic Galerkin formulation (SG) for shallow water equations (SWE)}\label{Sec2_2}

In this subsection, we review existing results on the hyperbolicity-preserving stochastic Galerkin formulation of shallow water equations (SGSWE) \cite{dai2021hyperbolicity}. We follow a standard Galerkin procedure in the stochastic space. It begins with reducing the problem to an alternative finite-dimensional form by replacing the solutions $(h,q)$ by the ansatz,
\begin{equation}\label{K-term-PCE}
\begin{split}
  h \simeq h_P &\coloneqq \sum_{k \in [K]} \widehat{h}_k(x,t)\phi_k(\xi),\\
   q \simeq q_P &\coloneqq \sum_{k \in [K]} \widehat{q}_k(x,t)\phi_k(\xi),
\end{split}
\end{equation}
respectively, and the bottom $B$ by $\Pi_P[B]$ defined in \eqref{PCE-truncated}, where notation $[K] \coloneqq \{1,2,...,K\}$. With a special choice of how the Galerkin projection is applied to the nonlinear, non-polynomial term $q^2/h$ introduced in \cite{dai2021hyperbolicity}, a SG system of balance laws was derived, with state variables being the coefficients in \eqref{K-term-PCE}
\begin{equation}\label{SGSWE}
    \widehat{U}_t + \widehat{F}(\widehat{U})_x = \widehat{S}(\widehat{U},\widehat{B}),
\end{equation}
where $\widehat{U} = \big(  \widehat{h}^\top, \widehat{q}^\top  \big)^\top \in \mathbb{R}^{2K}$, and $\widehat{h} , \widehat{q}$ are length-$K$ vectors whose entries are the coefficients in \eqref{K-term-PCE}. The flux and source term are defined by
\begin{align}\label{SG_flux_source}
    \widehat{F}(\widehat{U}) & = \begin{pmatrix}
        \widehat{q} \\ \mathcal{P}(\widehat{q})\mathcal{P}^{-1}(\widehat{h})\widehat{q} + \frac{1}{2}g \mathcal{P}(\widehat{h})\widehat{h}
    \end{pmatrix}, & \widehat{S}(\widehat{U},\widehat{B}) & = \begin{pmatrix}
        0 \\ -g\mathcal{P}(\widehat{h})\widehat{B}_x
    \end{pmatrix}.
\end{align}
In the deterministic case, the flux and source term \eqref{SG_flux_source} reduce to their corresponding deterministic forms. Recall the deterministic SWE is hyperbolic under the condition of positive water height ($h>0$). There is a natural extension of this property to the SG system \eqref{SGSWE}.

\begin{theorem}[Theorem 3.1 in \cite{dai2021hyperbolicity}]\label{hyperbolicitySG}
    If the matrix $\mathcal{P}(\widehat{h})$ is strictly positive definite at every point $(x,t)$ in the computational spatial-temporal domain, then the SG formulation \eqref{SGSWE} is hyperbolic.
\end{theorem}
This result is proven by identifying a symmetrizing similarity transform of the SG SWE flux Jacobians $\frac{\partial \widehat{F}}{\partial \widehat{U}}$. This flux Jacobian will be useful for us later, so we explicitly provide it below: When $\mathcal{P}(\widehat{h})$ is invertible, the well-defined velocity 
\begin{align}\label{velocity_coeff}
  \widehat{u} &\coloneqq  \mathcal{P}^{-1}(\widehat{h})\widehat{q},
\end{align}
is a stochastic representation of the velocity variable, i.e., this term can be interpreted as the vectors of the PCE coefficients of the velocity $u \coloneqq q/h$. The flux Jacobian of the SG SWE system \eqref{SGSWE} can then be expressed in terms of $K \times K$ blocks as follows:
\begin{equation}\label{SG_Jacobian}
        \frac{\partial \widehat{F}}{\partial \widehat{U}} = \begin{pmatrix}
            0& I \\
            g\mathcal{P}(\widehat{h}) - \mathcal{P}(\widehat{q})\mathcal{P}^{-1}(\widehat{h})\mathcal{P}(\widehat{u}) & \mathcal{P}(\widehat{q})\mathcal{P}^{-1}(\widehat{h}) + \mathcal{P}(\widehat{u})   \\ 
        \end{pmatrix}.
\end{equation}

\subsection{Skew-symmetric stochastic Galerkin formulation for shallow water equations (SGSWE)}

In this subsection, we introduce a skew-symmetric form of the SGSWE, which also preserves hyperbolicity. However, the skew-symmetric SGSWE is not a standard conservation/balance law, which motivates the following two definitions.

\begin{definition}[Hyperbolicity for first-order PDE]\label{def_hyper}
    A first-order system of partial differential equations is said to be hyperbolic if it can be written, via an invertible linear transformation, in the quasilinear form
\begin{equation}\label{quasilinear_system1}
    U_t + A(U) U_x = S(U),
\end{equation}
where the transformation consists of left multiplication by an invertible matrix $L(U)$ that depends only on the state variable and involves no spatial or temporal derivatives. The system \eqref{quasilinear_system1} is called hyperbolic if the matrix $A(U)$ is diagonalizable with real eigenvalues \cite{lax1973hyperbolic,leveque2002finite}.
\end{definition}

\begin{definition}
    For an arbitrary first-order (in time and space) system of partial differential equations, it has an associated companion balance law if it can be transformed to a conservation law 
    \begin{equation}\label{conservation_system1}
        U_t + F(U)_x = S(U),
    \end{equation}
    where the transformation consists of left multiplication by an invertible matrix $L(U)$ that depends only on the state variable and involves no spatial or temporal derivatives. Assume $U$ is smooth, and it has an associated companion balance law. We say \eqref{conservation_system1} has a companion balance law if there exists an entropy flux pair $\big(E(U), H(U)\big)$ satisfying
    \begin{equation}
        E(U)_t + H(U)_x = 0,
    \end{equation}
   where $E(U)$ is convex on $U$.
\end{definition}

\begin{remark}
    In the deterministic case, the skew-symmetric SWE is
    \begin{equation}\label{SWE_skew1}
    \begin{aligned}
       &  h_t + (hu)_x  &&= 0 \\
   & (hu)_t + (hu^2)_x + gh(h+b)_x - ( \frac{1}{2}h_tu + \frac{1}{2}(hu)_x u) && = 0,
    \end{aligned}
\end{equation}
which can be obtained by applying a linear operator $L:\mathbb{R}^2 \to \mathbb{R}^2$ to 
\begin{equation}\label{SWE_1}
    \begin{aligned}
      &  h_t + q_x& &= 0\\
      &  q_t + (\frac{1}{2}gh^2 + hu^2)_x  & &= -ghb_x,
    \end{aligned}
\end{equation}
where the matrix $L(U) = \begin{pmatrix}
    1 & 0 \\-\frac{1}{2}u & 1
\end{pmatrix},$ which is invertible. Hence, the skew-symmetric SWE \eqref{SWE_skew1} is hyperbolic if \eqref{SWE_1} is hyperbolic, which is true when $h>0$. In addition, let $V \coloneqq [g(h+b) - \frac{1}{2}u^2, u]^\top$ be the entropy variable. Multiplying \eqref{SWE_1} by $V^\top$, we can recover the companion balance law
\begin{align}\label{SWE_companionbl}
    E(U)_t + H(U)_x = 0,
\end{align}
where $E(U) = \frac{1}{2}(hu^2 + gh^2) + ghb$, and $H(U) = \frac{1}{2}hu^3 + ghu(h+b).$ Let the modified entropy variable $\Tilde{V} \coloneqq   (L^\top)^{-1}V = [g(h+b),u]^\top.$ Multiplying the skew symmetric SWE \eqref{SWE_skew1} by $\Tilde{V}^\top$, we can also recover the same companion balance law \eqref{SWE_companionbl}. 

\textbf{Note:} $\Tilde{V} \neq (\frac{\partial E}{\partial (LU)})^\top$, since $L$ may be related to $U$. In the case of SWE, $L$ involves velocity $u$, which is related to the state variable $U$. Instead, we can define $\Tilde{V} =  (L^\top)^{-1}V =  (L^\top)^{-1}(\frac{\partial E}{\partial U})^\top$. It follows that
\begin{equation}
    \Tilde{V}^\top (LU_t) = \frac{\partial E}{\partial U}(L)^{-1}LU_t = \frac{\partial E}{\partial U} U_t = E_t.
\end{equation}

\end{remark}

We can then derive the skew-symmetric SGSWE, starting from the skew-symmetric stochastic SWE:

\begin{equation}\label{skewsym_stochasticSWE}
    \begin{aligned}
    &h_t(x,t,\xi) + q_x(x,t,\xi) & = 0,\\
    & q_t(x,t,\xi) + (hu^2)_x(x,t,\xi) + gh(h+B)_x(x,t,\xi) - \frac{1}{2}h_tu(x,t,\xi) - \frac{1}{2}q_xu(x,t,\xi) & =  0,
\end{aligned}
\end{equation}
where $u \coloneqq \frac{q}{h}$. This skew-symmetric system \eqref{skewsym_stochasticSWE} can be derived by multiplying the first equation of \eqref{stochasticSWE} by $\frac{u}{2}$, and subtracting it from the second equation of \eqref{stochasticSWE}. The skew-symmetric form of the shallow water equations has been widely employed in recent studies to derive entropy-stable DG discretizations. In this work, we extend the idea of \cite{fu2022high} to the stochastic case, which requires great effort and more details.

We first consider stochastic discretization on \eqref{skewsym_stochasticSWE}, which is based on the idea of the hyperbolicity-preserving stochastic Galerkin formulation of shallow water equations \cite{dai2021hyperbolicity}. We follow a standard Galerkin procedure in the stochastic space, which starts from reducing the problem to an alternative finite-dimensional form by replacing the solutions $(h,q)$ by the ansatz,
\begin{align}\label{sol-PCE}
    h\simeq h_{s}(x,t,\xi) &= \sum_{j \le K} \widehat{h}_{j}(x,t)\phi_j(\xi) &   q\simeq q_{s}(x,t,\xi) &= \sum_{j \le K} \widehat{q}_{j}(x,t)\phi_j(\xi),
\end{align}
where $\widehat{h}_{j}, \widehat{q}_{j}$ are the truncated PCE coefficients of the SG solutions depending on the physical variable $x$ and the time variable $t$. In addition, the bottom $B$ is replaced by $\Pi_P[B]$. With a special choice of how the Galerkin projection is applied to the nonlinear, non-polynomial term $q^2/h$ introduced in \cite{dai2021hyperbolicity}, a SG system of balance laws was derived, whose state variables are the coefficients in \eqref{sol-PCE}

\begin{equation}\label{SGSWE_skew}
    \begin{aligned}
         & \widehat{h}_t + \widehat{q}_x & &= 0\\
      & \widehat{q}_t + \left( \mathcal{P}(\widehat{q})\widehat{u} \right)_x + g\mathcal{P}(\widehat{h})(\widehat{h}+\widehat{B})_x - \frac{1}{2}\mathcal{P}(\widehat{h}_t)\widehat{u} - \frac{1}{2}\mathcal{P}(\widehat{q}_x)\widehat{u} & &=0,
    \end{aligned}
\end{equation}
where $\widehat{u} \coloneqq \mathcal{P}^{-1}(\widehat{h})\widehat{q}$. In the deterministic case, the stochastic Galerkin system \eqref{SGSWE_skew} reduces to the deterministic skew-symmetric SWE \eqref{SWE_skew1}, and the hyperbolicity condition reduces to $h$ being positive. There
is a natural extension of this property to the SG formulation of the SWE. 
\begin{theorem}[Hyperbolicity of skew-symmetric SGSWE]\label{hyperbolicitySG_skew}
    If the matrix $\mathcal{P}(\widehat{h})$ is strictly positive definite at every point $(x,t)$ in the computational spatial-temporal domain, then the SG formulation \eqref{SGSWE_skew} is hyperbolic.
\end{theorem}
\begin{proof}
    Since the skew-symmetric SGSWE \eqref{SGSWE_skew} is not in the standard form of conservation/balance laws, to discuss its hyperbolicity, we need to recall the definition of hyperbolicity for first-order PDE in Definition \ref{def_hyper}. The skew-symmetric SGSWE \eqref{SGSWE_skew} can be obtained by applying a linear operator $\widehat{L}$ to the SGSWE \eqref{SGSWE}. Here $\widehat{L}: \mathbb{R}^{2K} \to \mathbb{R}^{2K}$,
$$\widehat{L}\begin{pmatrix}
    \widehat{x} \\ \widehat{y}
\end{pmatrix} \coloneqq \begin{pmatrix}
    \widehat{x} \\ \widehat{y}-\frac{1}{2}\mathcal{P}(\widehat{x})\widehat{u}
\end{pmatrix}, \quad \widehat{x}, \widehat{y} \in \mathbb{R}^{K}.$$ Equivalently, the matrix of $\widehat{L}$ is defined by
\begin{equation} 
\widehat{L}_M = 
    \begin{pmatrix}
       I & 0 \\ -\frac{1}{2}\mathcal{P}(\widehat{u}) & I
    \end{pmatrix}.
\end{equation}
Since $\widehat{L}$ is an invertible linear operator, the skew-symmetric SGSWE \eqref{SGSWE_skew} is hyperbolic if \eqref{SGSWE} is hyperbolic, which is true when $\mathcal{P}(\widehat{h})$ is symmetric positive definite.
\end{proof}

\begin{remark}\label{RK2}
    Let $\widehat{V} \coloneqq [-\frac{1}{2}\widehat{u}^\top\mathcal{P}(\widehat{u}) + g(\widehat{h}+\widehat{B})^\top,\widehat{u}^\top]^\top$. Multiplying the SGSWE \eqref{SGSWE} by $\widehat{V}^\top,$ we can recover the stochastic companion balance law
\begin{equation}\label{SGSWE_companionbl}
    \widehat{E}(\widehat{U})_t + \widehat{H}(\widehat{U})_x = 0,
\end{equation}
where $\widehat{E}(\widehat{U}) = \frac{1}{2}(\widehat{q}^\top \widehat{u} + g\| 
\widehat{h} \|^2) + g\widehat{h}^\top\widehat{B},$ and $\widehat{H}(\widehat{U}) = \frac{1}{2}\widehat{u}^\top \mathcal{P}(\widehat{q})\widehat{u} + g\widehat{q}^\top(\widehat{h}+\widehat{B}).$ In addition, let the modified stochastic entropy variable $\Tilde{\widehat{V}} \coloneqq (\widehat{L}^*)^{-1}\widehat{V} = [g(\widehat{h}+\widehat{B})^\top, \widehat{u}^\top]^\top.$ Multiplying the skew symmetric SGSWE \eqref{SGSWE_skew} by $\left(\Tilde{\widehat{V}}\right)^\top$, we can also recover the same stochastic companion balance law \eqref{SGSWE_companionbl}.\\

\end{remark}

\section{Entropy conservative, entropy stable, and well-balanced discontinuous Galerkin formulation for skew-symmetric SGSWE}\label{Sec3}

\subsection{Discontinuous Galerkin--stochastic Galerkin formulation for skew-symmetric stochastic shallow water equations}
In this section, we derive a discontinuous Galerkin--stochastic Galerkin (DG-SG) formulation for the skew-symmetric stochastic shallow water equations. In this section and the rest of this paper, we use the subscript $\Delta x$ for quantities after spatial discretization.

Let the physical domain $\Omega$ be periodic, partitioned into elements 
$\mathcal{I}_i = [x_{i-1/2}, x_{i+1/2}]$. Let $\Delta x_i \coloneqq |\mathcal{I}_i|$ denote the size of each element. Let $\{\mathcal{E}_i\}$ denote the set of all cell interfaces. For any polynomial degree $k \ge 0$, let
\begin{equation}
    V_{\Delta x}^k \coloneqq \{ v \in L^2(\Omega): v|_{\mathcal{I}_i} \in \mathbb{P}_k(\mathcal{I}_i), \forall \mathcal{I}_i \subset \Omega\},
\end{equation}
where $\mathbb{P}_k(\mathcal{I}_i)$ is the space of polynomials of degree at most $k$ on the element $\mathcal{I}_i$. Let ${\bf V}_{\Delta x}^k$ be the vectorized version of the space $V_{\Delta x}^k$. At each interface $x_{i+1/2} = \mathcal{I}_i \cap \mathcal{I}_{i+1}$, we define average and jump quantities by

\begin{align}\label{average_jump}
   & \{ a\}_{i+1/2} \coloneqq \frac{a|_{i+1/2}^+ + a|_{i+1/2}^-}{2}, & \llbracket a \rrbracket_{i+1/2} \coloneqq a|_{i+1/2}^+ - a|_{i+1/2}^-,
\end{align}
where $a|_{i+1/2}^+ \coloneqq \lim_{x\to x_{i+1/2}^+} a|_{\mathcal{I}_{i+1}}(x)$ and $a|_{i+1/2}^- \coloneqq \lim_{x\to x_{i+1/2}^-} a|_{\mathcal{I}_{i}}(x)$.

The semi-discrete DG-SG formulation of \eqref{SGSWE_skew} reads as follows: find DG solutions $(\widehat{h}_{\Delta x}, \widehat{q}_{\Delta x}, \widehat{u}_{\Delta x}) \in {\bf V}_{\Delta x}^k$, such that
\begin{equation}\label{DG-SGSWE_eq1_velocity_correction}
\begin{aligned}
\sum_{\mathcal{I}_i \subset \Omega} \int_{\mathcal{I}_i} (\widehat{h}_{\Delta x, j})_t \widehat{e}_{\Delta x,j}\,dx
- \sum_{\mathcal{I}_i \subset \Omega} \int_{\mathcal{I}_i} \widehat{q}_{\Delta x, j} (\widehat{e}_{\Delta x,j})_x\,dx
- \sum_{\mathcal{E}_i} (\widehat{f}_q)_j  \llbracket \widehat{e}_{\Delta x,j}\rrbracket
= 0,
\end{aligned}
\end{equation}
for all $\widehat{e}_{\Delta x,j} \in V_{\Delta x}^k$, $j=1,\dots,K$, and
\begin{equation}\label{DG-SGSWE_eq2_velocity_correction}
\begin{aligned}
&\sum_{\mathcal{I}_i \subset \Omega} \int_{\mathcal{I}_i} (\widehat{q}_{\Delta x, j})_t \widehat{v}_{\Delta x,j}\,dx
- \sum_{\mathcal{I}_i \subset \Omega} \int_{\mathcal{I}_i} \big(\mathcal{P}(\widehat{q}_{\Delta x})\widehat{u}_{\Delta x}\big)_j (\widehat{v}_{\Delta x,j})_x\,dx
- \sum_{\mathcal{E}_i} (\widehat{f}_{qu})_j  \llbracket \widehat{v}_{\Delta x,j}\rrbracket \\
&+ \sum_{\mathcal{I}_i \subset \Omega} \int_{\mathcal{I}_i} g\big(\mathcal{P}(\widehat{h}_{\Delta x}+\widehat{B}_{\Delta x})_x \widehat{v}_{\Delta x,j}\widehat{h}_{\Delta x} \big)_j\,dx 
+ \sum_{\mathcal{E}_i} \Big(g \mathcal{P}(\llbracket \widehat{h}_{\Delta x} + \widehat{B}_{\Delta x} \rrbracket)
\{\widehat{v}_{\Delta x,j}\widehat{h}_{\Delta x}\} \Big)_j \\
&- \sum_{\mathcal{I}_i \subset \Omega} \int_{\mathcal{I}_i} \frac{1}{2}\big(\mathcal{P}(\widehat{h}_{\Delta x})_t \widehat{u}_{\Delta x} \big)_j \widehat{v}_{\Delta x,j}\,dx 
+ \sum_{\mathcal{I}_i \subset \Omega} \int_{\mathcal{I}_i} \frac{1}{2} \Big( \mathcal{P}(\widehat{q}_{\Delta x})
(\widehat{v}_{\Delta x,j} \widehat{u}_{\Delta x})_x \Big)_j\,dx \\
&+ \sum_{\mathcal{E}_i} \frac{1}{2} \Big( \mathcal{P}(\widehat{f}_{q})
\llbracket \widehat{v}_{\Delta x,j} \widehat{u}_{\Delta x}\rrbracket  \Big)_j
= 0,
\end{aligned}
\end{equation}
for all $\widehat{v}_{\Delta x,j}\in V_{\Delta x}^k$, $j=1,\dots,K$, and
\begin{equation}\label{DG-SGSWE_eq3_velocity_correction}
\sum_{\mathcal{I}_i \subset \Omega} \int_{\mathcal{I}_i}
(\mathcal{P}(\widehat{h}_{\Delta x})\widehat{u}_{\Delta x} - \widehat{q}_{\Delta x})\widehat{w}_{\Delta x,j}\,dx
= 0,
\quad \forall \widehat{w}_{\Delta x,j}\in V_{\Delta x}^k,\; j=1,\dots,K.
\end{equation}

Here, $\widehat{h}_{\Delta x}(x,t)$ and $\widehat{q}_{\Delta x}(x,t)$ are vector-valued functions (the PCE coefficients vectors), whose $j-$th components are denoted by $\widehat{h}_{\Delta x, j}(x,t)$ and $\widehat{q}_{\Delta x, j}(x,t)$, respectively. The last equation \eqref{DG-SGSWE_eq3_velocity_correction} is introduced to define the velocity update $\widehat{u}_{\Delta x}$ such that $\widehat{q}_{\Delta x}$ is the $L^2$-projection of $\mathcal{P}(\widehat{h}_{\Delta x})\widehat{u}_{\Delta x}$ onto the DG approximation space $\bm{V}_{\Delta x}^k$. This relation allows $\widehat{u}_{\Delta x}$ to be used as a test function in the proposed DG-SG scheme. Moreover, \eqref{DG-SGSWE_eq3_velocity_correction} plays a key role in the derivation of the discrete entropy equation. The numerical fluxes $\widehat{f}_{q}$ and $\widehat{f}_{qu}$ will be specified later to ensure that the resulting DG-SG scheme satisfies the desired physical properties, including the well-balanced property and entropy stability.

\subsection{Entropy conservative, entropy stable, and well-balanced numerical fluxes}
In this subsection, we establish the entropy stability and well-balanced property of the proposed DG-SG formulation with appropriate choices of numerical fluxes. We start by defining the well-balanced property of a numerical scheme in the stochastic context, which ensures that stochastic ``lake-at-rest'' steady states are equilibrium states at the discrete level.

\begin{definition}[Well-Balanced DG-SG scheme for stochastic SWE]\label{def:wb}
The solution $(h_{\Delta x, s}, q_{\Delta x, s})$ of the proposed scheme \eqref{DG-SGSWE_eq1_velocity_correction}--\eqref{DG-SGSWE_eq3_velocity_correction} is said to be well-balanced if it satisfies the stochastic ``lake-at-rest'' solution,
\begin{align}\label{well-balance}
    q_{\Delta x, s}(x,t,\xi)  &\equiv 0,& h_{\Delta x, s}(x,t,\xi) + B_{\Delta x}(x,t,\xi)&\equiv C(\xi),
\end{align}
where $C(\xi)$ is a random scalar depending only on $\xi$.
  Such a well-balanced solution describes a still water surface with a flat but stochastic water surface. In terms of the system of PCE coefficients, \eqref{well-balance} is equivalent to the following vector equations
\begin{align}\label{well-balance-vector}
    \widehat{q}_{\Delta x} & \equiv {\bf 0}, & \widehat{h}_{\Delta x}+\widehat{B}_{\Delta x} &\equiv \widehat{C},& \forall (x,t) &\in \Omega \times [0,T],
\end{align}
with spatial domain $\Omega$ and terminal time $T$. Note that the vector equations \eqref{well-balance-vector} represent a steady state of the skew-symmetric SG SWE \eqref{SGSWE_skew}.
\end{definition}
Many geophysical flows, including lake perturbations and tsunami propagation, involve small deviations from the lake-at-rest equilibrium of the shallow water equations. In such regimes, it is essential for numerical schemes to be well-balanced to preserve steady states and avoid generating spurious waves that can dominate the true dynamics. In addition, to ensure physical admissibility of nonlinear weak solutions, it is also necessary to impose an entropy condition. We recover the stochastic companion balance law \eqref{SGSWE_companionbl} from the skew-symmetric SGSWE \eqref{SGSWE_skew}, and we therefore expect the semi-discrete numerical scheme to inherit the semi-discrete entropy conservative/stable property in a discrete sense.

\begin{definition}[Entropy conservative/stable DG-SG scheme]\label{def_ECES}
      The semi-discrete DG-SG solutions defined in \eqref{DG-SGSWE_eq1_velocity_correction}--\eqref{DG-SGSWE_eq3_velocity_correction}, with corresponding numerical fluxes, are said to be entropy conservative or stable if,
      \begin{align}\label{def_eq_ECES}
         \frac{d}{dt}E_{\Delta x} = 0, && \text{or,}& &\frac{d}{dt}E_{\Delta x} \le 0, 
      \end{align}
    respectively, where
\begin{equation}\label{entropy_DGSGSWE_velocity_correction}
    E_{\Delta x} = \sum_{\mathcal{I}\subset\Omega} \int_{\mathcal{I}} \Big(\frac{1}{2}g \| \widehat{h}_{\Delta x}\|^2_{l^2} + g\widehat{h}_{\Delta x}^\top\widehat{B}_{\Delta x} + \frac{1}{2}\widehat{q}_{\Delta x}^\top\widehat{u}_{\Delta x} \Big) dx.
\end{equation}
\end{definition}
Our main result is that we can ensure the entropy conservative/stable property through special selection of the numerical fluxes $\widehat{f}_q, \widehat{f}_{qu}$.
\begin{theorem}[Entropy conservative/stable and well-balanced DG-SG scheme]\label{EC/ES_thm}
   Suppose that the bottom topography function $B$ is independent of time. The semi-discrete DG-SG scheme \eqref{DG-SGSWE_eq1_velocity_correction}--\eqref{DG-SGSWE_eq3_velocity_correction} for the skew-symmetric SG SWE system \eqref{SGSWE_skew}, is well-balanced, and,
   \begin{itemize}
       \item entropy conservative, with the central fluxes
       \begin{align}\label{EC_flux}
            \widehat{f}_q &= \{ \mathcal{P}(\widehat{h}_{\Delta x}) \widehat{u}_{\Delta x}\},&   \widehat{f}_{qu} =  \mathcal{P}(\widehat{f}_q) \{\widehat{u}_{\Delta x}\} &= \mathcal{P}(\{ \mathcal{P}(\widehat{h}_{\Delta x}) \widehat{u}_{\Delta x}\})\{\widehat{u}_{\Delta x}\},
       \end{align}
       \item entropy stable, with the Lax-Friedrichs type numerical fluxes
       \begin{align}\label{ES_flux}
            \widehat{f}_q  &= \{\mathcal{P}(\widehat{h}_{\Delta x}) \widehat{u}_{\Delta x} \}  - \frac{1}{2}\alpha_1 \llbracket \widehat{h}_{\Delta x} + \widehat{B}_{\Delta x} \rrbracket, &  \widehat{f}_{qu}  &= \mathcal{P}(\{ \mathcal{P}(\widehat{h}_{\Delta x}) \widehat{u}_{\Delta x}\})\{ \widehat{u}_{\Delta x}\}  - \frac{1}{2}\alpha_1 \llbracket \mathcal{P}(\widehat{h}_{\Delta x} + \widehat{B}_{\Delta x})\widehat{u}_{\Delta x} \rrbracket, 
       \end{align}
       for any $\alpha_1 \ge 0$.
   \end{itemize}
\end{theorem}
Before proving the entropy stability, we first introduce two lemmas:
\begin{lemma}[Commutative dot product]
    For any symmetric matrix $A \in \mathbb{R}^{k\times k}$ and vectors $\vec{b}, \vec{c} \in \mathbb{R}^k, \forall k \in \mathbb{N}$, we have 
    \begin{equation}\label{eq_lm1}
        A\vec{b} \cdot \vec{c} = A \vec{c} \cdot \vec{b}.
    \end{equation}
\end{lemma}

\begin{lemma}[A property of $\mathcal{P}$]
    For any $\widehat{a}, \widehat{b}$, we have the equality
    \begin{equation}\label{eq_lm2}
        \frac{1}{2}\Big( \mathcal{P}(\widehat{a})\widehat{b}^+ \cdot \widehat{b}^+ -  \mathcal{P}(\widehat{a})\widehat{b}^- \cdot \widehat{b}^-\Big) = \mathcal{P}(\widehat{a})\{ \widehat{b}\} \cdot \llbracket \widehat{b} \rrbracket.
    \end{equation}
\end{lemma}

\begin{proof}[Proof of Theorem \ref{EC/ES_thm}]
    It begins with constructing the entropy equation by taking $\widehat{e}_{\Delta x,j} = g(\widehat{h}_{\Delta x, j}+\widehat{B}_{\Delta x,j}), \forall j$ in \eqref{DG-SGSWE_eq1_velocity_correction} and taking $\widehat{v}_{\Delta x,j} = \widehat{u}_{\Delta x,j}, \forall j$ in \eqref{DG-SGSWE_eq2_velocity_correction}, and adding the equations \eqref{DG-SGSWE_eq1_velocity_correction} and \eqref{DG-SGSWE_eq2_velocity_correction} together, then we get the following equation

    \begin{equation}\label{entropy_eq_DGSGSWE}
        \begin{aligned}
            &\sum_{\mathcal{I}_i \subset \Omega}  \sum_{j=1}^{K}  \int_{\mathcal{I}_i} (\widehat{h}_{\Delta x, j})_t g (\widehat{h}_{\Delta x, j}+ \widehat{B}_{\Delta x,j})dx 
+  \sum_{\mathcal{I}_i \subset \Omega}  \sum_{j=1}^{K} 
\int_{\mathcal{I}_i} (\widehat{q}_{\Delta x, j})_t \widehat{u}_{\Delta x,j} dx \\ 
 & -  \sum_{\mathcal{I}_i \subset \Omega}  \sum_{j=1}^{K} \int_{\mathcal{I}_i} \Big( \frac{1}{2}\mathcal{P}(\widehat{h}_{\Delta x})_t \widehat{u}_{\Delta x} \Big)_j \widehat{u}_{\Delta x,j}dx\\
  \coloneqq &  I_1 + I_2,
        \end{aligned}
    \end{equation}
where we combine all time-derivative terms in the left-hand side (LHS) of \eqref{entropy_eq_DGSGSWE}, and in the right-hand side (RHS) of \eqref{entropy_eq_DGSGSWE}, we denote spatial terms without gravitational terms by $I_1$  and those with gravitational terms by $I_2$ as follows,
\begin{equation}
    \begin{aligned}
        I_1 \coloneqq & \sum_{\mathcal{I}_i \subset \Omega} \sum_{j=1}^{K} \int_{\mathcal{I}_i} \Big(\mathcal{P}(\widehat{q}_{\Delta x})\widehat{u}_{\Delta x}\Big)_j (\widehat{u}_{\Delta x,j})_xdx - \sum_{\mathcal{I}_i \subset \Omega} \sum_{j=1}^{K} \int_{\mathcal{I}_i} \frac{1}{2} \Big( \mathcal{P}(\widehat{q}_{\Delta x})(\widehat{u}_{\Delta x,j}\widehat{u}_{\Delta x})_x \Big)_j dx \\
        & + \Big( \sum_{\mathcal{E}_i} \sum_{j=1}^{K} (\widehat{f}_{qu})_j  \llbracket \widehat{u}_{\Delta x,j}\rrbracket - \sum_{\mathcal{E}_i} \sum_{j=1}^{K} \frac{1}{2} \big( \mathcal{P}(\widehat{f}_q) \llbracket \widehat{u}_{\Delta x,j}\widehat{u}_{\Delta x} \rrbracket \big)_j  \Big),\\
    \end{aligned}
\end{equation}
and,
\begin{equation}
    \begin{aligned}
         I_2 \coloneqq & \sum_{\mathcal{I}_i \subset \Omega} \sum_{j=1}^{K} \int_{\mathcal{I}_i} \widehat{q}_{\Delta x, j} \big( g(\widehat{h}_{\Delta x, j}+\widehat{B}_{\Delta x,j}) \big)_x dx 
        - \sum_{\mathcal{I}_i \subset \Omega} \sum_{j=1}^{K} \int_{\mathcal{I}_i} \Big( g\mathcal{P}(\widehat{h}_{\Delta x}+\widehat{B}_{\Delta x})_x \widehat{u}_{\Delta x,j}\widehat{h}_{\Delta x} \Big)_j dx \\
        & - \sum_{\mathcal{E}_i} \sum_{j=1}^{K} \Big( g\mathcal{P}(\llbracket \widehat{h}_{\Delta x} + \widehat{B}_{\Delta x} \rrbracket) \{ \widehat{u}_{\Delta x,j} \widehat{h}_{\Delta x} \} \Big)_j
        + \sum_{\mathcal{E}_i} \sum_{j=1}^{K} (\widehat{f}_q)_j \llbracket g(\widehat{h}_{\Delta x, j}+\widehat{B}_{\Delta x,j}) \rrbracket. \\
    \end{aligned}
\end{equation}
    The LHS of \eqref{entropy_eq_DGSGSWE} is evaluated to obtain the derivative of semi-discrete entropy in \eqref{entropy_DGSGSWE_velocity_correction}:
\begin{equation}
\begin{aligned}
&\sum_{\mathcal{I}_i \subset \Omega}  \sum_{j=1}^{K}  \int_{\mathcal{I}_i} (\widehat{h}_{\Delta x, j})_t g (\widehat{h}_{\Delta x, j}+ \widehat{B}_{\Delta x,j})dx 
+  \sum_{\mathcal{I}_i \subset \Omega}  \sum_{j=1}^{K} 
\int_{\mathcal{I}_i} (\widehat{q}_{\Delta x, j})_t \widehat{u}_{\Delta x,j} dx  \\
&\quad -  \sum_{\mathcal{I}_i \subset \Omega}  \sum_{j=1}^{K} \int_{\mathcal{I}_i} \Big( \frac{1}{2}\mathcal{P}(\widehat{h}_{\Delta x})_t \widehat{u}_{\Delta x} \Big)_j \widehat{u}_{\Delta x,j}dx\\
\overset{\sum_{j=1}^K}{=} & \sum_{\mathcal{I}_i \subset \Omega} \int_{\mathcal{I}_i} g(\widehat{h}_{\Delta x})_t \cdot \widehat{h}_{\Delta x} dx 
+ \sum_{\mathcal{I}_i \subset \Omega} \int_{\mathcal{I}_i} g(\widehat{h}_{\Delta x})_t \cdot \widehat{B}_{\Delta x} dx 
+ \sum_{\mathcal{I}_i \subset \Omega} \int_{\mathcal{I}_i} (\widehat{q}_{\Delta x})_t \cdot \widehat{u}_{\Delta x} dx\\
& - \sum_{\mathcal{I}_i \subset \Omega} \int_{\mathcal{I}_i} \frac{1}{2}\mathcal{P}(\widehat{h}_{\Delta x})_t\widehat{u}_{\Delta x} \cdot \widehat{u}_{\Delta x}dx\\
\overset{\eqref{DG-SGSWE_eq3_velocity_correction}}{=} & \sum_{\mathcal{I}_i \subset \Omega} \int_{\mathcal{I}_i} g(\widehat{h}_{\Delta x})_t \cdot \widehat{h}_{\Delta x} dx 
+ \sum_{\mathcal{I}_i \subset \Omega} \int_{\mathcal{I}_i} g(\widehat{h}_{\Delta x})_t \cdot \widehat{B}_{\Delta x} dx 
+ \sum_{\mathcal{I}_i \subset \Omega} \int_{\mathcal{I}_i} \frac{1}{2}(\widehat{q}_{\Delta x})_t \cdot \widehat{u}_{\Delta x} dx\\
&+ \sum_{\mathcal{I}_i \subset \Omega}\int_{\mathcal{I}_i} \frac{1}{2}\mathcal{P}(\widehat{h}_{\Delta x})_t\widehat{u}_{\Delta x} \cdot \widehat{u}_{\Delta x} dx 
+ \sum_{\mathcal{I}_i \subset \Omega} \int_{\mathcal{I}_i} \frac{1}{2}\mathcal{P}(\widehat{h}_{\Delta x})(\widehat{u}_{\Delta x})_t\cdot \widehat{u}_{\Delta x}dx \\
&\quad - \sum_{\mathcal{I}_i \subset \Omega} \int_{\mathcal{I}_i} \frac{1}{2}\mathcal{P}(\widehat{h}_{\Delta x})_t\widehat{u}_{\Delta x} \cdot \widehat{u}_{\Delta x}dx\\
\overset{\eqref{DG-SGSWE_eq3_velocity_correction},\eqref{eq_lm1}}{=} & \sum_{\mathcal{I}_i \subset \Omega} \int_{\mathcal{I}_i} g(\widehat{h}_{\Delta x})_t \cdot \widehat{h}_{\Delta x} dx 
+ \sum_{\mathcal{I}_i \subset \Omega} \int_{\mathcal{I}_i} g(\widehat{h}_{\Delta x})_t \cdot \widehat{B}_{\Delta x} dx 
+ \sum_{\mathcal{I}_i \subset \Omega} \int_{\mathcal{I}_i} \frac{1}{2}(\widehat{q}_{\Delta x})_t \cdot \widehat{u}_{\Delta x} dx\\
& + \sum_{\mathcal{I}_i \subset \Omega} \int_{\mathcal{I}_i} \frac{1}{2}\widehat{q}_{\Delta x} \cdot (\widehat{u}_{\Delta x})_t dx\\
\overset{\eqref{entropy_DGSGSWE_velocity_correction}}{=} & \frac{d}{dt}E_{\Delta x}.
\end{aligned}
\end{equation}
The terms $I_1$ and $I_2$ can be simplified as follows.
\begin{equation}\label{I1_entropy_eq_DGSGSWE_volocity_correction}
    \begin{aligned}
        I_1 = & \sum_{\mathcal{I}_i \subset \Omega} \sum_{j=1}^{K} \int_{\mathcal{I}_i} \Big(\mathcal{P}(\widehat{q}_{\Delta x})\widehat{u}_{\Delta x}\Big)_j (\widehat{u}_{\Delta x,j})_xdx - \sum_{\mathcal{I}_i \subset \Omega} \sum_{j=1}^{K} \int_{\mathcal{I}_i} \frac{1}{2} \Big( \mathcal{P}(\widehat{q}_{\Delta x})(\widehat{u}_{\Delta x,j}\widehat{u}_{\Delta x})_x \Big)_j dx \\
        & + \Big( \sum_{\mathcal{E}_i} \sum_{j=1}^{K} (\widehat{f}_{qu})_j  \llbracket \widehat{u}_{\Delta x,j}\rrbracket - \sum_{\mathcal{E}_i} \sum_{j=1}^{K} \frac{1}{2} \big( \mathcal{P}(\widehat{f}_q) \llbracket \widehat{u}_{\Delta x,j}\widehat{u}_{\Delta x} \rrbracket \big)_j  \Big)\\
        \overset{\sum_{j=1}^K}{=} &  \sum_{\mathcal{I}_i \subset \Omega} \int_{\mathcal{I}_i} \mathcal{P}(\widehat{q}_{\Delta x})\widehat{u}_{\Delta x} \cdot (\widehat{u}_{\Delta x})_x -  \sum_{\mathcal{I}_i \subset \Omega} \int_{\mathcal{I}_i} \frac{1}{2}\mathcal{P}(\widehat{q}_{\Delta x})\widehat{u}_{\Delta x} \cdot (\widehat{u}_{\Delta x})_x -  \sum_{\mathcal{I}_i \subset \Omega} \int_{\mathcal{I}_i} \frac{1}{2}\mathcal{P}(\widehat{q}_{\Delta x})(\widehat{u}_{\Delta x})_x\cdot \widehat{u}_{\Delta x} dx\\
        & +  \sum_{\mathcal{E}_i} \Big(\widehat{f}_{qu} \cdot \llbracket \widehat{u}_{\Delta x} \rrbracket  - \big(\frac{1}{2}\mathcal{P}(\widehat{f}_{q})\widehat{u}_{\Delta x}^+  \cdot \widehat{u}_{\Delta x}^+ - \frac{1}{2}\mathcal{P}(\widehat{f}_{q})\widehat{u}_{\Delta x}^-  \cdot \widehat{u}_{\Delta x}^-  \big) \Big)\\
         \overset{\eqref{eq_lm2}}{=} &   \sum_{\mathcal{E}_i} \Big(\widehat{f}_{qu} \cdot \llbracket \widehat{u}_{\Delta x} \rrbracket  - \mathcal{P}(\widehat{f}_{q})\{ \widehat{u}_{\Delta x} \} \cdot \llbracket \widehat{u}_{\Delta x} \rrbracket  \Big).
    \end{aligned}
\end{equation}

\begin{equation}\label{I2_entropy_eq_DGSGSWE_velocity_correction}
    \begin{aligned}
        I_2 = & \sum_{\mathcal{I}_i \subset \Omega} \sum_{j=1}^{K} \int_{\mathcal{I}_i} \widehat{q}_{\Delta x, j} \big( g(\widehat{h}_{\Delta x, j}+\widehat{B}_{\Delta x,j}) \big)_x dx 
        - \sum_{\mathcal{I}_i \subset \Omega} \sum_{j=1}^{K} \int_{\mathcal{I}_i} \Big( g\mathcal{P}(\widehat{h}_{\Delta x}+\widehat{B}_{\Delta x})_x \widehat{u}_{\Delta x,j}\widehat{h}_{\Delta x} \Big)_j dx \\
        & - \sum_{\mathcal{E}_i} \sum_{j=1}^{K} \Big( g\mathcal{P}(\llbracket \widehat{h}_{\Delta x} + \widehat{B}_{\Delta x} \rrbracket) \{ \widehat{u}_{\Delta x,j} \widehat{h}_{\Delta x} \} \Big)_j
        + \sum_{\mathcal{E}_i} \sum_{j=1}^{K} (\widehat{f}_q)_j \llbracket g(\widehat{h}_{\Delta x, j}+\widehat{B}_{\Delta x,j}) \rrbracket \\
        \overset{\eqref{average_jump}}{=} & \sum_{\mathcal{I}_i \subset \Omega}  \int_{\mathcal{I}_i} g\widehat{q}_{\Delta x} \cdot (\widehat{h}_{\Delta x}+\widehat{B}_{\Delta x})_x dx 
        - \sum_{\mathcal{I}_i \subset \Omega} \int_{\mathcal{I}_i} g\mathcal{P}\big((\widehat{h}_{\Delta x}+\widehat{B}_{\Delta x})_x\big)\widehat{h}_{\Delta x} \cdot \widehat{u}_{\Delta x} dx\\
        & - \sum_{\mathcal{E}_i} \frac{1}{2} g\mathcal{P}(\llbracket \widehat{h}_{\Delta x} + \widehat{B}_{\Delta x}  \rrbracket) \widehat{h}_{\Delta x}^+\cdot \widehat{u}_{\Delta x}^+
        - \sum_{\mathcal{E}_i} \frac{1}{2} g\mathcal{P}(\llbracket \widehat{h}_{\Delta x} + \widehat{B}_{\Delta x}  \rrbracket) \widehat{h}_{\Delta x}^-\cdot \widehat{u}_{\Delta x}^-\\
        & + \sum_{\mathcal{E}_i} g \widehat{f}_q \cdot \llbracket \widehat{h}_{\Delta x} + \widehat{B}_{\Delta x} \rrbracket \\
        \overset{\eqref{eq_lm1},\eqref{eq_lm2}, \eqref{commutative_prop}}{=} & \sum_{\mathcal{I}_i \subset \Omega}  \int_{\mathcal{I}_i} g\widehat{q}_{\Delta x} \cdot (\widehat{h}_{\Delta x}+\widehat{B}_{\Delta x})_x dx 
        - \sum_{\mathcal{I}_i \subset \Omega} \int_{\mathcal{I}_i} g\mathcal{P}\big( \widehat{h}_{\Delta x} \big) \widehat{u}_{\Delta x} \cdot (\widehat{h}_{\Delta x}+\widehat{B}_{\Delta x})_x dx\\
        & - \sum_{\mathcal{E}_i} \{g\mathcal{P}(\widehat{h}_{\Delta x})\widehat{u}_{\Delta x}\} \cdot \llbracket \widehat{h}_{\Delta x}+ \widehat{B}_{\Delta x} \rrbracket 
        +  \sum_{\mathcal{E}_i} g \widehat{f}_q \cdot \llbracket \widehat{h}_{\Delta x} + \widehat{B}_{\Delta x} \rrbracket \\
        \overset{\eqref{DG-SGSWE_eq3_velocity_correction}}{=} & \sum_{\mathcal{E}_i} g\Big( \widehat{f}_q - \{\mathcal{P}(\widehat{h}_{\Delta x}) \widehat{u}_{\Delta x}\}  \Big)\cdot \llbracket \widehat{h}_{\Delta x} + \widehat{B}_{\Delta x} \rrbracket.
    \end{aligned}
\end{equation}
To make the proposed scheme entropy conservative or entropy stable \eqref{def_eq_ECES}, it suffices to select appropriate numerical fluxes $\widehat{f}_q, \widehat{f}_{qu}$ to let $I_1$ and $I_2$ both vanish or non-positive. If we take 
\begin{equation}\label{numericalflux_ES_1_velocity_correction}
    \widehat{f}_q  = \{\mathcal{P}(\widehat{h}_{\Delta x}) \widehat{u}_{\Delta x} \} - \frac{1}{2}\alpha_1 \llbracket \widehat{h}_{\Delta x} + \widehat{B}_{\Delta x} \rrbracket, 
\end{equation}
for some $\alpha_1 \ge 0$, $I_2$ becomes a diffusive term
\begin{equation}
    I_2 = -\frac{g}{2} \alpha_1 \sum_{\mathcal{E}_i} \Big\| \llbracket\widehat{h}_{\Delta x} + \widehat{B}_{\Delta x} \rrbracket \Big\|^2_{l^2} \le 0.
\end{equation}
In addition, by substituting \eqref{numericalflux_ES_1_velocity_correction} into \eqref{I1_entropy_eq_DGSGSWE_volocity_correction}, we have
\begin{equation}\label{I1_evaluation}
\begin{aligned}
     I_1 = & \sum_{\mathcal{E}_i} \Big( \big(\widehat{f}_{qu}   - \mathcal{P}(\widehat{f}_{q})\{ \widehat{u}_{\Delta x} \} \big) \cdot \llbracket \widehat{u}_{\Delta x} \rrbracket  \Big)\\
     = & \sum_{\mathcal{E}_i} \Big( \big(\widehat{f}_{qu}   - \mathcal{P}(\{ \mathcal{P}(\widehat{h}_{\Delta x}) \widehat{u}_{\Delta x}\})\{ \widehat{u}_{\Delta x} \} + \frac{1}{2}\alpha_1\mathcal{P}(\llbracket \widehat{h}_{\Delta x} + \widehat{B}_{\Delta x} \rrbracket) \{ \widehat{u}_{\Delta x} \} \big) \cdot \llbracket \widehat{u}_{\Delta x} \rrbracket  \Big).
\end{aligned}
\end{equation}
To balance the contribution of $\widehat{f}_q$ in $I_1$, we can further take
\begin{equation}\label{numericalflux_ES_2_velocity_correction}
    \widehat{f}_{qu}  = \mathcal{P}(\{ \mathcal{P}(\widehat{h}_{\Delta x}) \widehat{u}_{\Delta x}\})\{ \widehat{u}_{\Delta x}\} - \frac{1}{2}\alpha_1 \llbracket \mathcal{P}(\widehat{h}_{\Delta x} + \widehat{B}_{\Delta x})\widehat{u}_{\Delta x} \rrbracket, 
\end{equation}
and substitute it into \eqref{I1_evaluation} and directly expand the jump and average quantities by definition \eqref{average_jump}. Note that since only the spatial derivative of the bottom is involved in the continuous system and the DG-SG discretization, the vertical datum is chosen so that the discrete bottom representation satisfies
\begin{align}
    B_{\Delta x}(x,\xi)\ge 0,
\end{align}
for all $x\in\Omega$ and for $\rho$-a.e. $\xi$. Under this assumption, $P(\{\widehat h_{\Delta x}+\widehat B_{\Delta x}\})$ is symmetric positive definite, since the hyperbolicity-preserving condition gives $P(\{\widehat h_{\Delta x}\})> 0$, while $B_{\Delta x}\ge 0$ implies $P(\{\widehat B_{\Delta x}\})\ge 0$.  Hence, $I_1$ also becomes dissipative:
\begin{equation}
\begin{aligned}
     I_1 = &-\frac{1}{2}\alpha_1 \sum_{\mathcal{E}_i} \mathcal{P}(\{ \widehat{h}_{\Delta x} + \widehat{B}_{\Delta x} \})\llbracket \widehat{u}_{\Delta x} \rrbracket \cdot \llbracket \widehat{u}_{\Delta x} \rrbracket, \\
     \le & -\frac{1}{2}\alpha_1 \sum_{\mathcal{E}_i} \lambda_{1} \Big\| \llbracket \widehat{u}_{\Delta x} \rrbracket \Big\|^2_{l^2},
\end{aligned}
\end{equation}
where $\lambda_1$ is the smallest eigenvalue of the symmetric positive definite matrix $\mathcal{P}(\{ \widehat{h}_{\Delta x} + \widehat{B}_{\Delta x}\})$. Therefore, $\frac{d}{dt}E_{\Delta x} = I_1+I_2 \le 0$, which shows the entropy stability property. Note that when $\alpha \equiv 0$ in $I_1$ and $I_2$, the DG-SG scheme will reduce to an entropy conservative numerical scheme.

Finally, to establish the well-balanced property, it suffices to show that, for the initial data
\begin{align}\label{wb_initial}
    \widehat{u}_{\Delta x} &\equiv \mathbf{0}, 
    & \widehat{h}_{\Delta x}+\widehat{B}_{\Delta x} &= \text{const vector},
    & \forall x &\in \Omega,\quad t=0,
\end{align}
the time derivatives of all PCE coefficients vanish, i.e.,
\begin{align}
    (\widehat{h}_{\Delta x, j})_t = 0,
\qquad
(\widehat{q}_{\Delta x, j})_t = 0,
\qquad \forall j=1,\ldots,K.
\end{align}

Under the condition \eqref{wb_initial}, all spatial volume and numerical flux terms in the DG-SG scheme vanish. In particular, combining \eqref{DG-SGSWE_eq1_velocity_correction} with \eqref{wb_initial} yields $(\widehat{h}_{\Delta x, j})_t = 0,
\forall j=1,\ldots,K.$ Substituting this result into \eqref{DG-SGSWE_eq2_velocity_correction} and again using \eqref{wb_initial}, we further obtain $(\widehat{q}_{\Delta x, j})_t = 0,
\forall j=1,\ldots,K.$ Therefore, the semi-discrete solution preserves the discrete steady state, which proves the well-balanced property \eqref{well-balance-vector}.

\end{proof}
\begin{remark}[Adaptive diffusion]
We may introduce different diffusion parameters in the two numerical fluxes by incorporating an adaptive diffusion quantity in $\widehat{f}_{qu}$. Suppose that we take the first numerical flux in \eqref{numericalflux_ES_1_velocity_correction} and consider $\widehat{f}_{qu}$ adaptively as
\begin{equation}\label{numericalflux_ES_2_adap_velocity_correction}
    \widehat{f}_{qu}  = \mathcal{P}( \{\mathcal{P}(\widehat{h}_{\Delta x}) \widehat{u}_{\Delta x}  \})\{ \widehat{u}_{\Delta x}\}- \frac{1}{2}\alpha_1 \llbracket W_1 \rrbracket - \frac{1}{2}(\alpha_2 - \alpha_1)\llbracket W_2 \rrbracket,
\end{equation}
which implies 
\begin{equation}\label{RK_I1}
\begin{aligned}
      I_1 =& -\sum_{\mathcal{E}_i}\frac{1}{2}\alpha_1  \Big( \llbracket W_1 \rrbracket - \mathcal{P}(\llbracket \widehat{h}_{\Delta x} + \widehat{B}_{\Delta x} \rrbracket) \{\widehat{u}_{\Delta x} \} \Big) \cdot \llbracket \widehat{u}_{\Delta x} \rrbracket, \\
      & - \sum_{\mathcal{E}_i} (\frac{1}{2}\alpha_2 - \frac{1}{2}\alpha_1)\llbracket W_2 \rrbracket \cdot \llbracket \widehat{u}_{\Delta x} \rrbracket.
\end{aligned}
\end{equation}
To make the first term of $I_1$ in \eqref{RK_I1} perform as a diffusive term, we can take $W_1 = \mathcal{P}(\widehat{h}_{\Delta x}+\widehat{B}_{\Delta x})\widehat{u}_{\Delta x}$. For the second term, a straightforward way is to take $W_2 = \mathrm{sign}(\alpha_2 - \alpha_1)(\widehat{u}_{\Delta x})$ to make it a diffusive term. With these choices, we have an adaptive second numerical flux \eqref{numericalflux_ES_2_adap_velocity_correction}, resulting in
\begin{equation}
    \begin{aligned}
        I_1 = & -\frac{1}{2}\alpha_1 \sum_{\mathcal{E}_i} \mathcal{P}(\{ \widehat{h}_{\Delta x} + \widehat{B}_{\Delta x} \})\llbracket \widehat{u}_{\Delta x} \rrbracket \cdot \llbracket \widehat{u}_{\Delta x} \rrbracket - \frac{1}{2}|\alpha_2 - \alpha_1| \sum_{\mathcal{E}_i} \Big\| \llbracket \widehat{u}_{\Delta x} \rrbracket \Big\|^2_{l^2},\\
        \le &  -\sum_{\mathcal{E}_i} (\frac{1}{2}\alpha_1 \lambda_{1} + \frac{1}{2}|\alpha_2 - \alpha_1|) \Big\| \llbracket \widehat{u}_{\Delta x} \rrbracket \Big\|^2_{l^2}.
    \end{aligned}
\end{equation}    
\end{remark}
In the context of the shallow water equations, we use the terminologies entropy and energy interchangeably, as they represent the same convex functional associated with the system. Accordingly, we will refer to entropy conservation/stability and energy conservation/stability with the same meaning throughout this work.

\section{Algorithmic details and pseudocode}\label{Sec4}
In this section, we present the algorithm details of the corresponding fully discrete schemes. The complete algorithmic pseudocode is given in algorithm \ref{alg_DG-SG}. In the following, we describe several key components required to ensure robustness and stability in implementation.

\subsection{Desingularization}
To avoid the singularity in the velocity reconstruction when the water height approaches zero, we employ a desingularization procedure in one spatial dimension. Instead of directly computing the velocity as $u = q/h$, we define the discrete velocity through a regularized relation based on the PCE coefficients. 

More precisely, let $\widehat{h}_{\Delta x}$ and $\widehat{q}_{\Delta x}$ denote the vectors of the PCE coefficients of the water height and discharge, respectively. We define the DG velocity $\widehat{u}_{\Delta x}$ by \eqref{DG-SGSWE_eq3_velocity_correction}
\begin{equation}
\sum_{\mathcal{I}_i \subset \Omega} \int_{\mathcal{I}_i}
(\mathcal{P}(\widehat{h}_{\Delta x})\widehat{u}_{\Delta x} - \widehat{q}_{\Delta x})\widehat{w}_{\Delta x,j}\,dx
= 0,
\quad \forall \widehat{w}_{\Delta x,j},\; j=1,\dots,K.
\end{equation}
Hence, the computation of $\widehat{u}_{\Delta x}$ requires $\big(\mathcal{P}(\widehat{h}_{\Delta x})\big)^{-1}$, which could be ill-conditioned in some situations, for example, when the water height is too small in the near-dry region. To address this issue, we adopt a desingularization procedure introduced in \cite{kurganov2007second}, whose stochastic variant was introduced in \cite{dai2021hyperbolicity,dai2022hyperbolicity,dai2024energy,epshteyn2025energy}.

Let $\mathcal{P}(\widehat{h}) = Q^\top \Pi Q,$ be the eigenvalue decomposition of $\mathcal{P}(\widehat{h})$, where $\Pi = \mathrm{diag}(\lambda_1,\dots,\lambda_{K})$. For each eigenvalue $\lambda_i$, $i=1,\ldots,K$, and a given positive constant $\epsilon$, we define
\begin{align}
    \Pi^{cor} \coloneqq \mathrm{diag}(\lambda_1^{cor},\dots,\lambda_{K}^{cor}),& & 
\lambda_i^{cor} \coloneqq \frac{\sqrt{\lambda_i^4+\max(\lambda_i^4,\epsilon^4)}}{\sqrt{2}\lambda_i}.
\end{align}
In the numerical experiments, we take $\epsilon=\Delta x_{\mathcal{I}}$ on each element $\mathcal{I}$. We then define $\mathcal{P}(\widehat{h}_{\Delta x})^{cor} \coloneqq Q^\top \Pi^{cor} Q$ and substitute it into the velocity update formula in our DG-SG scheme:
\begin{equation}\label{vel_update}
     \sum_{\mathcal{I}_i \subset \Omega} \int_{\mathcal{I}_i} (\mathcal{P}(\widehat{h}_{\Delta x})^{cor}\widehat{u}_{\Delta x} - \widehat{q}_{\Delta x} ) \widehat{w}_{\Delta x,j}\,dx =  0, \quad \forall \widehat{w}_{\Delta x,j}, \ j=1,\dots,K,
\end{equation}
To maintain consistency of the scheme, the discharge $\widehat{q}_{\Delta x}$ needs to be subsequently recomputed by
\begin{equation}
    \sum_{\mathcal{I}_i \subset \Omega} \int_{\mathcal{I}_i} (\mathcal{P}(\widehat{h}_{\Delta x})\widehat{u}_{\Delta x} - \widehat{q}_{\Delta x} ) \widehat{w}_{\Delta x,j}\,dx =  0, \quad \forall \widehat{w}_{\Delta x,j}, \ j=1,\dots,K.
\end{equation}

\subsection{Hyperbolicity-preserving criterion}\label{sec_hpc}
The hyperbolicity of the skew-symmetric SGSWE \eqref{SGSWE_skew} requires that $\mathcal{P}(\widehat{h})$ be positive-definite, i.e., $\mathcal{P}(\widehat{h})>0$. In our DG-SG semi-discrete scheme, this corresponds to the condition $\mathcal{P}(\widehat{h}_{\Delta x})>0$ holding for each DG-quadrature point in each cell involved in the scheme. While this may be true at the current time, additional conditions must be imposed to ensure that it is valid at the next time step. A sufficient condition, \cite[Theorem 3.2]{dai2021hyperbolicity}, was introduced to make it applicable. Based on this, we define the hyperbolicity-preserving property for a fully discrete numerical scheme as follows.

\begin{definition}
    For the DG-SGSWE scheme, provided 
    \begin{equation}\label{pp_creterion_assumption}
        h_{\Delta x, s}(x_i,t^n,\xi_j) \coloneqq \sum_{k \le K} \widehat{h}_{\Delta x,k}(x_i,t^n)\phi_k(\xi_j) >0, \quad \forall i=1,...,N_q, j=1,...,M.
    \end{equation}
    where $\{x_i\}$ is the set of DG quadrature points involved in the scheme, $t^n$ is the current time step, $\{ \xi_j\}$ are stochastic nodes such that the $M-$point positivity quadrature rule is exact on $P_{\Lambda}^3$.
    Then the DG-SGSWE scheme is said to preserve positivity of the water height, if  \begin{equation}\label{pp_creterion_condition}
        h_{\Delta x, s}(x_i,t^{n+1},\xi_j)>0, \quad \forall i=1,...,N_q, j=1,...,M.
    \end{equation}
\end{definition}

Suppose that \eqref{pp_creterion_assumption} is satisfied and the forward Euler method is applied. To enforce \eqref{pp_creterion_condition} at the next time level, we use one of the following two strategies.
\begin{remark}[Mean-positivity-preserving and linear scaling]
Let $\Phi(\xi) \coloneqq \big( \phi_1(\xi), \phi_2(\xi),... \phi_K(\xi) \big)^\top$. Suppose at each element $\mathcal{I}$ with size $\Delta x_{\mathcal{I}}$, $\Delta t$ satisfies
    \begin{equation}
\Delta t^n < \Delta t_{\Delta x_{\mathcal{I}}}^n \coloneqq \min_{j=1,...,M} \Big( \Delta x_{\mathcal{I}} \Big| 
 \frac{(\overline{\widehat{h}}_{\Delta x}^{n})^\top\Phi(\xi_j)}{(\widehat{f_q}|_{\mathcal{I}^+}^n - \widehat{f_q}|_{\mathcal{I}^-}^n)^\top \Phi(\xi_j)} \Big| \Big),
 \end{equation}
 for sufficiently large $M$, s.t., the $M-$point quadrature rule that is exact on $P^3_{\Lambda}$ on the stochastic space. Then rescale the water height at each quadrature points in physical space ($S_\mathcal{I}$ denotes the set of involved points) involved in the DG-SG scheme at the element $\mathcal{I}$ as follows:
 \begin{equation}
 \tilde{h}_{\Delta x, s}(x,\xi_j)^n = \theta\big(h_{\Delta x, s}(x,\xi_j)^n - \overline{h}_{\Delta x, s}(\xi_j)^n \big) + \overline{h}_{\Delta x, s}(\xi_j), \quad \theta = \min\{1, \overline{h}_{\Delta x, s}(\xi_j)^n \oslash  (\overline{h}_{\Delta x, s}(\xi_j)^n - m_{\mathcal{I}}) \},
 \end{equation}
 where $m_{\mathcal{I}}= \min_{x\in S_{\mathcal{I}}} h_{\Delta x, s}^n(x,\xi_j)$, and $\oslash$ is the elementwise division. This linear scaling is extended from \cite{xing2010positivity} to fit our DG-SG formulation.
\end{remark}

\begin{remark}[Straightforward restriction]
    Compared to the first way, one can make a more straightforward restriction on $\Delta t$ as follows:
        \begin{equation}
\Delta t^n < \Delta t_{\Delta x_{\mathcal{I}}}^n \coloneqq \min_{i =1,...,|S_{\mathcal{I}}|,}\min_{j=1,...,M} \Big( \Big| 
\frac{\widehat{h}_{\Delta x}^{n}(x_i)^\top\Phi(\xi_j)}{\mathcal{L}(\widehat{h}_{\Delta x}^n)_i^\top \Phi(\xi_j)} \Big| \Big).
\end{equation}
 Here $|S_{\mathcal{I}}|$ denotes the size of the set $S_{\mathcal{I}}$, consisting of involved physical quadrature points in the DG-SG scheme, $\mathcal{L}(\widehat{h}_{\Delta x}^n)_i \coloneqq \mathcal{L}(\widehat{h}_{\Delta x}^n)|_{x_i}$ is defined by rewriting the DG-SG scheme with forward Euler time discretization on the variable $\widehat{h}_{\Delta x}$ in the form of $\widehat{h}_{\Delta x}^{n+1} = \widehat{h}_{\Delta x}^{n} +\Delta t^n \mathcal{L}(\widehat{h}_{\Delta x}^n).$  
\end{remark}

\subsection{The troubled-cell indicator and characteristic TVB limiter}\label{sec_slopelimiting}
In order to suppress spurious oscillations near solution discontinuities, a slope limiter procedure is needed. We follow the standard slope limiting procedure in RKDG methods \cite{cockburn2001runge,qiu2005comparison}, which consists of two steps:
\begin{itemize}
    \item First, we identify the troubled cells, i.e., the elements that need the limiting procedure.
    \item Second, we apply the corresponding slope limiter to the solutions in troubled cells
\end{itemize}

First, we extend the scaling modification of Fu-Shu troubled-cell indicator proposed in \cite{fu2022high} to a stochastic version to identify troubled cells. The troubled-cell indicator for a discontinuous function $p \in V_{\Delta x}^k$ is defined as follows:
\begin{equation}
    I_{\mathcal{I}}(p) = \frac{\sum_{T \in \omega (\mathcal{I})}| \overline{\overline{p}}_T - \overline{p}_{\mathcal{I}}|}{\overline{p}_{\max} - \overline{p}_{\min}},
\end{equation}
where $\omega(\mathcal{I})$ is the union of cells sharing a common edge with the element $\mathcal{I}$, and $p|_T$ is the discontinuous function defined on $T\in \omega(\mathcal{I})$, and $\overline{\overline{p}}_T$ is the cell-average of the polynomial $p|_T$ extended to the cell $\mathcal{I}$ from its neighboring cell $T$, and $\overline{p}_{\max}$ and $\overline{p}_{\min}$ are the global maximal and minimal cell average on the domain. For a given tolerance $tol>0$, the cell $\mathcal{I}$ is marked as a troubled cell if 
$I_{\mathcal{I}}(p)>tol$. In our DG-SG scheme, we take the stochastic mean of $h_{\Delta x,s}+B_{\Delta x}$ as the indicating function. 

Second, in those detected troubled cells, we apply the characteristic-wise TVB limiter \cite{cockburn2001runge,cockburn1998runge} on the equilibrium variables $\big(\widehat{h}_{\Delta x}^\top + \widehat{B}_{\Delta x}^\top, \widehat{q}_{\Delta x}^\top \big)^\top$ in our DG-SG scheme. The use of the water surface $\widehat h_{\Delta x}+\widehat B_{\Delta x}$, rather than $\widehat h_{\Delta x}$ alone, is compatible with the well-balanced structure and avoids unnecessary limiting near the lake-at-rest steady state.

\subsection{High order SSP-RK fully discretization with adaptive time step size}
The timestep restriction for the positivity-preserving criterion in Section \ref{sec_hpc} is derived under the assumption of a forward Euler time discretization. This restriction can be extended to higher-order strong stability-preserving Runge–Kutta (SSP-RK) methods, which may be written as convex combinations of forward Euler steps with multiple intermediate stages \cite{gottlieb2001strong}. We employ the adaptive time-stepping strategy proposed in \cite{chertock2015well-adaptiveRK}, which permits the application of forward Euler-based timestep restrictions within SSP-RK schemes. In practice, the timestep constraint is updated at each intermediate stage of the Runge–Kutta procedure.


\subsection{Augmented energy}\label{def_aug_energy}
Definition \ref{def_ECES} characterizes energy conservative and energy stable schemes based on properties satisfied in the \textit{interior} of the computational domain. However, this definition does not, in general, guarantee global energy conservation or stability in the presence of boundary contributions. In particular, boundary fluxes may introduce energy into the domain, leading to an increase in the total energy. Such behavior is not attributable to a deficiency of the numerical scheme, but rather reflects the underlying physical dynamics of the model. This phenomenon may arise even for standard implementations with outflow boundary conditions; see, e.g., \cite{epshteyn2025energy}.

With the same idea of removing the boundary effect when computing the entropy/energy in our DG-SG solution, we define the augmented energy $\widetilde{E}_{\Delta x}$ over the one-spatial computational domain divided into $N$ elements as
\begin{equation}
\begin{aligned}
   \frac{d}{dt}\widetilde{E}_{\Delta x}(t)
   =& \sum_{i=1}^N \frac{d}{dt}\widetilde{E}_{\Delta x,i}(t) \\
   =& \sum_{i=1}^N \frac{d}{dt}E_{\Delta x,i}(t)
   - \textbf{boundary terms} \\
   \overset{\eqref{I1_entropy_eq_DGSGSWE_volocity_correction},\eqref{I2_entropy_eq_DGSGSWE_velocity_correction}}{=}& \sum_{i=1}^N \frac{d}{dt}E_{\Delta x,i}(t)
   - \bigg\{
   (\widehat{f}_{qu}\cdot \widehat{u}_{\Delta x})\big|_{x_L}
   -(\widehat{f}_{qu}\cdot \widehat{u}_{\Delta x})\big|_{x_R}  
   +\frac{1}{2}
   \big(\mathcal{P}(\widehat{f}_{q})\widehat{u}_{\Delta x}
   \cdot \widehat{u}_{\Delta x}\big)\big|_{x_R}
   -\frac{1}{2}
   \big(\mathcal{P}(\widehat{f}_{q})\widehat{u}_{\Delta x}
   \cdot \widehat{u}_{\Delta x}\big)\big|_{x_L}
   \bigg\} \\
   &+
   \frac{1}{2}g
   \big(
   \mathcal{P}(\llbracket
   \widehat{h}_{\Delta x}+\widehat{B}_{\Delta x}
   \rrbracket)
   \widehat{h}_{\Delta x}
   \cdot \widehat{u}_{\Delta x}
   \big)\big|_{x_L} 
   +
   \frac{1}{2}g
   \big(
   \mathcal{P}(\llbracket
   \widehat{h}_{\Delta x}+\widehat{B}_{\Delta x}
   \rrbracket)
   \widehat{h}_{\Delta x}
   \cdot \widehat{u}_{\Delta x}
   \big)\big|_{x_R} \\
   &+
   \big(
   g\widehat{f}_q\cdot
   (\widehat{h}_{\Delta x}+\widehat{B}_{\Delta x})
   \big)\big|_{x_R}
   -
   \big(
   g\widehat{f}_q\cdot
   (\widehat{h}_{\Delta x}+\widehat{B}_{\Delta x})
   \big)\big|_{x_L},
   \\
   \widetilde{E}_{\Delta x}(0)
   =& \sum_{i=1}^N E_{\Delta x,i}(0).
\end{aligned}
\end{equation}
where $x_L, x_R$ represent the left boundary and the right boundary, respectively. This definition allows us to assess the intrinsic energy behavior of the numerical scheme independently of boundary effects.

\begin{algorithm}[htbp]
\caption{Pseudocode of the fully discrete DG-SG scheme for stochastic SWE using SSPRK3}\label{alg_DG-SG}
\textbf{Input:} Initial data: $h_0, q_0$; Bottom topography $B$; Terminal time $T$; Polynomial index set $\Lambda$ for stochastic Galerkin projection; A basis of DG solution space.\\
\textbf{Initialize (SG):} Compute PCE coefficients $\widehat{h}, \widehat{q}, \widehat{B}$ as functions of $x$.\\
\textbf{Initialize (DG):} Compute DG approximation solutions $\widehat{h}_{\Delta x}, \widehat{u}_{\Delta x}, \widehat{q}_{\Delta x}$ and bottom $\widehat{B}_{\Delta x}$.\\
\textbf{Initialize time:} Set current time $t=0$. 
\begin{algorithmic}
\While{$t<T$}
\State \textbf{Stage 1 of RK3} (Input: $\widehat{h}_{\Delta x}, \widehat{q}_{\Delta x}, \widehat{u}_{\Delta x}$)
\State Compute the $\Delta t$ using the positivity-preserving criterion in Section \ref{sec_hpc} and CFL condition.
\State Compute $\widehat{h}_{\Delta x}^{(1)}$ by \eqref{DG-SGSWE_eq1_velocity_correction} using forward Euler with $\Delta t$.
\State Compute $\widehat{q}_{\Delta x}^{(1)}$ for all by \eqref{DG-SGSWE_eq2_velocity_correction} using forward Euler with $\Delta t$.
\State Apply the slope limiting procedure in Section \ref{sec_slopelimiting} to $\big(\widehat{h}_{\Delta x}^{(1)} + \widehat{B}_{\Delta x}, \widehat{q}_{\Delta x}^{(1)} \big)$
\State Compute $\widehat{u}_{\Delta x}^{(1)}$ using \eqref{DG-SGSWE_eq3_velocity_correction}.

\State \textbf{Stage 2 of RK3} (Input: $\widehat{h}_{\Delta x}^{(1)}, \widehat{q}_{\Delta x}^{(1)}, \widehat{u}_{\Delta x}^{(1)}$)
\State Compute the $\Delta t_2$ using the positivity-preserving criterion in Section \ref{sec_hpc} and CFL condition.
\If{$\Delta t_2 \ge \Delta t$}
\State Compute $\widehat{h}_{\Delta x}^{(2*)}$ by \eqref{DG-SGSWE_eq1_velocity_correction} using forward Euler $\Delta t$.
\State Compute $\widehat{q}_{\Delta x}^{(2*)}$ for all by \eqref{DG-SGSWE_eq2_velocity_correction} using forward Euler $\Delta t$.
\State Update $\widehat{h}_{\Delta x}^{(2)}, \widehat{q}_{\Delta x}^{(2)}$ using the coefficients for SSPRK3 state 2:
$$\widehat{h}_{\Delta x}^{(2)} =  \frac{3}{4}\widehat{h}_{\Delta x}+ \frac{1}{4}\widehat{h}_{\Delta x}^{(2*)}, \quad \widehat{q}_{\Delta x}^{(2)} =  \frac{3}{4}\widehat{q}_{\Delta x}+ \frac{1}{4}\widehat{q}_{\Delta x}^{(2*)}$$
\State Apply the slope limiting procedure in Section \ref{sec_slopelimiting} to $\big(\widehat{h}_{\Delta x}^{(2)} + \widehat{B}_{\Delta x}, \widehat{q}_{\Delta x}^{(2)} \big)$
\State Compute $\widehat{u}_{\Delta x}^{(2)}$ using \eqref{DG-SGSWE_eq3_velocity_correction} with inputs $\widehat{h}_{\Delta x}^{(2)}, \widehat{q}_{\Delta x}^{(2)}$.
\ElsIf{$\Delta t_2 < \Delta t$}
\State Go back to Stage 1 of RK3 and set $\Delta t = \Delta t_2$
\EndIf
\State \textbf{Stage 3 of RK3} (Input: $\widehat{h}_{\Delta x}^{(2)}, \widehat{q}_{\Delta x}^{(2)}, \widehat{u}_{\Delta x}^{(2)}$)
\State Compute the $\Delta t_3$ using the positivity-preserving criterion in Section \ref{sec_hpc} and CFL condition.
\If{$\Delta t_3 \ge \Delta t$}
\State Compute $\widehat{h}_{\Delta x}^{(3*)}$ by \eqref{DG-SGSWE_eq1_velocity_correction} using forward Euler $\Delta t$.
\State Compute $\widehat{q}_{\Delta x}^{(3*)}$ for all by \eqref{DG-SGSWE_eq2_velocity_correction} using forward Euler $\Delta t$.
\State Update $\widehat{h}_{\Delta x}^{(3)}, \widehat{q}_{\Delta x}^{(3)}$ using the coefficients for SSPRK3 state 3:
$$\widehat{h}_{\Delta x}^{(3)} =  \frac{1}{3}\widehat{h}_{\Delta x}+ \frac{2}{3}\widehat{h}_{\Delta x}^{(3*)}, \quad \widehat{q}_{\Delta x}^{(3)} =  \frac{1}{3}\widehat{q}_{\Delta x}+ \frac{2}{3}\widehat{q}_{\Delta x}^{(3*)}$$
\State Apply the slope limiting procedure in Section \ref{sec_slopelimiting} to $\big(\widehat{h}_{\Delta x}^{(3)} + \widehat{B}_{\Delta x}, \widehat{q}_{\Delta x}^{(3)} \big)$
\ElsIf{$\Delta t_3 < \Delta t$}
\State Go back to Stage 1 of RK3 and set $\Delta t = \Delta t_3$
\EndIf
\State Denote $\widehat{h}_{\Delta x} = \widehat{h}_{\Delta x}^{(3)}, \widehat{q}_{\Delta x} = \widehat{q}_{\Delta x}^{(3)}$ and update $\widehat{u}_{\Delta x}$ using \eqref{DG-SGSWE_eq3_velocity_correction} with $\widehat{h}_{\Delta x}, \widehat{q}_{\Delta x}$.
\State $t \gets t+\Delta t$
\EndWhile
\end{algorithmic}
\end{algorithm}

\section{Numerical experiments}\label{Sec5}
In this section, we assess the performance of the proposed DG-SG schemes through a sequence of numerical experiments. The examples are designed to verify the main structural properties established in the previous sections, including high-order accuracy, well-balanced preservation of stochastic lake-at-rest states, entropy stability, robustness near dry regions, and applicability to multivariate random inputs. Unless specifically stated, the following numerical setup is used.

Let $\xi$ be a one-dimensional random variable supported on $[-1,1]$ with Beta distribution
\[
\rho(\xi) \propto (1-\xi)^{\alpha}(1+\xi)^{\beta}, \quad \alpha,\beta>-1.
\]
The associated orthonormal polynomial basis is given by the Jacobi polynomials with parameters $(\alpha,\beta)$. The special case $\alpha=\beta=0$ corresponds to the uniform distribution on $[-1,1]$, with Legendre polynomials as the orthonormal basis. We also consider the case of a two-dimensional random variable $\xi = (\xi^{(1)}, \xi^{(2)})$, where $\xi^{(1)}, \xi^{(2)} \stackrel{\mathrm{iid}}{\sim} \mathrm{Beta}(\beta+1,\alpha+1)$. The joint density is given by $\rho(\xi) = \rho(\xi^{(1)})\rho(\xi^{(2)})$, and the corresponding orthonormal basis is obtained by tensorizing one-dimensional orthonormal polynomials.

Throughout this section, we set the gravitational constant to $g=9.812$ in the first two examples and to $g=1$ in the remaining examples. Unless otherwise stated, we employ $K_{\rm PCE}=4$ terms in the polynomial chaos expansion. For discontinuous problems, we implement the slope limiting procedure and take $tol = 0.02$ as the default parameter in our troubled-cell indicator. For visualization, we plot the water surface $w \coloneqq h + B$ rather than the conservative variable $h$. We impose periodic boundary conditions in the accuracy test of Section 5.1, wall boundary conditions in the well-balanced tests of Section 5.2, and outflow boundary conditions using zero-order extrapolation in the rest of examples.

To compare energy conservation and dissipation, we measure the relative change in energy for periodic boundary conditions and the relative change in augmented energy for non-periodic boundary conditions, introduced in Section \ref{def_aug_energy}:
\begin{align}\label{eq:rel-energies}
  \text{relative change in (original) energy} &= \frac{E_{\Delta x}(t)- E_{\Delta x}(0)}{E_{\Delta x}(0)},& & \text{ for periodic boundary} \\ 
  \text{relative change in augmented energy} &= \frac{\widetilde{E}_{\Delta x}(t) - \widetilde{E}_{\Delta x}(0)}{\widetilde{E}_{\Delta x}(0)},& & \text{ for non-periodic boundary}
\end{align}
where $E_{\Delta x}(t) = \sum_{i}  E_{\Delta x,i}(t) =  \sum_{i} \int_{\mathcal{I}_i} \Big(\frac{1}{2}g \| \widehat{h}_{\Delta x}\|^2_{l^2} + g\widehat{h}_{\Delta x}^\top\widehat{B}_{\Delta x} + \frac{1}{2}\widehat{q}_{\Delta x}^\top\widehat{u}_{\Delta x} \Big) dx$. 

\subsection{Accuracy test}\label{eg_accuracytest}
We begin by examining the order of accuracy of the proposed DG-SG scheme for the stochastic shallow water system. This test is a stochastic modification of the test discussed in \cite{fu2022high}.

Consider the smooth problem with a stochastic initial water surface.
\begin{equation}\label{accuracytest_SG}
    B(x) = \sin^2(\pi x), \quad h(x,0,\xi) = 5+e^{\cos(2\pi x)} + 0.1\xi, \quad q(x,0) = \sin(\cos(2\pi x)),
\end{equation}
where $\xi$ follows the uniform distribution on $[-1,1]$. The computational domain is a periodic unit interval $[0,1]$, and the final time is $t=0.1$. 
 A reference solution is computed on a fine grid with 320 elements for each scheme. For the time evolution solver, we utilize the third-order Strong Stability-Preserving (SSP) Runge-Kutta (RK3) method \cite{gottlieb2001strong}. 

 We illustrate the order of accuracy of our schemes using the water height $h$ by computing the error between the reference solution and the numerical test solution with the $ L^2$ norm in physical space and the weighted $ L^2_{\rho}$ norm in stochastic space.
\begin{equation}\label{error_def}
\begin{split}
    \text{Error}(h_{\Delta x, s}) &= \| h_{\Delta x, s}(x,t,\xi)-h_{ref}(x,t,\xi)\|_{L^2(\Omega;L_{\rho}^2(\mathbb{R}^d))},\\
    \|h(x,t,\xi )\|_{L^2(\Omega;L_{\rho}^2(\mathbb{R}^d))} & \coloneqq \sqrt{\int_{\Omega}\|h(x,t,\xi)\|_{L_{\rho}^2}^2dx }
\end{split}
\end{equation}

We compute the convergence orders of the numerical solutions for different polynomial degrees, as summarized in Table \ref{table1}. In this smooth test case, no slope limiting is applied. The results demonstrate that the DG schemes achieve the expected theoretical orders of accuracy for all polynomial degrees. The relative energy evolution shown in Figure \ref{fig1} is also consistent with theoretical predictions. Despite the use of an energy-conservative (EC) numerical flux, the fully discrete DG scheme exhibits energy decay due to temporal discretization. As the time step decreases, the numerical solution approaches the energy-conservative limit.

\begin{table}[H]
\centering
\begin{tabular}{l|l|ll|ll}
\noalign{\smallskip}
\hline
\noalign{\smallskip}
& & \multicolumn{2}{c|}{$t=0.09$} & \multicolumn{2}{c}{$t=0.1$} \\
\hline
Scheme & Grid size & Error & Order & Error & Order \\
\hline
\multirow{4}{4em}{$K_{\rm DG}=1$}
& $20$  &2.1060e-02 &  &  2.7132e-02 &  \\
& $40$  & 5.2827e-03  & 1.9951 &  6.8268e-03 & 1.9907   \\
& $80$  & 1.0850e-03 & 2.2836 &   1.5494e-03  & 2.1395 \\
& $160$ & 1.9148e-04  & 2.5024&  2.6830e-04 & 2.5297 \\
\hline
\multirow{4}{4em}{$K_{\rm DG}=2$}
& $20$  & 2.3581e-03  &  &  3.8999e-03 &  \\
& $40$  & 3.0620e-04 & 2.9451 &  5.3899e-04 &  2.8551 \\
& $80$  & 2.9930e-05 &3.3548 &  5.6016e-05 & 3.2663 \\
& $160$ & 3.4251e-06 & 3.1274 &   5.6957e-06 & 3.2979\\
\hline
\multirow{4}{4em}{$K_{\rm DG}=3$}
& $20$  & 5.1125e-04 &  &  9.1144e-04 &  \\
& $40$  & 3.0818e-05 & 4.0522 &  8.0658e-05&    3.4983 \\
& $80$  &  1.7763e-06 & 4.1168 &  4.1137e-06& 4.2933 \\
& $160$ & 1.0971e-07 & 4.0171 &   2.4117e-07  & 4.0924 \\
\hline
\end{tabular}
\captionsetup{font={scriptsize},justification=raggedright}
\caption{Accuracy test for Section \ref{eg_accuracytest}. Orders of convergence at $t=0.09$ and $t=0.1$. The reference solution is computed with $320$ elements. $C_{\rm CFL}=0.1$. Error is computed using \eqref{error_def}.}
\label{table1}
\end{table}

\begin{figure}[htbp]
    \centering
    \begin{subfigure}[t]{0.5\textwidth}
        \centering
        \includegraphics[height=2.3in]{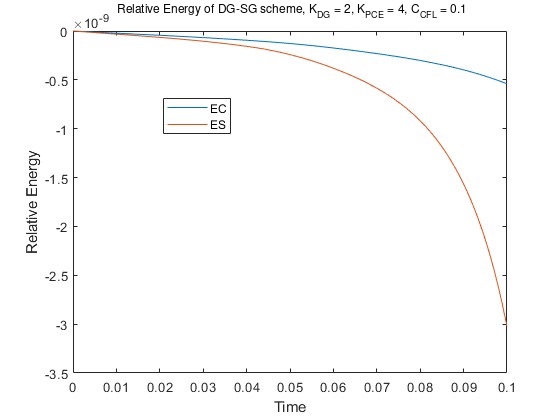}
        \captionsetup{font={scriptsize},justification=raggedright}
        \caption{EC compared to ES, $C_{\rm CFL}=0.1$}
    \end{subfigure}%
    ~ 
    \begin{subfigure}[t]{0.5\textwidth}
        \centering
        \includegraphics[height=2.3in]{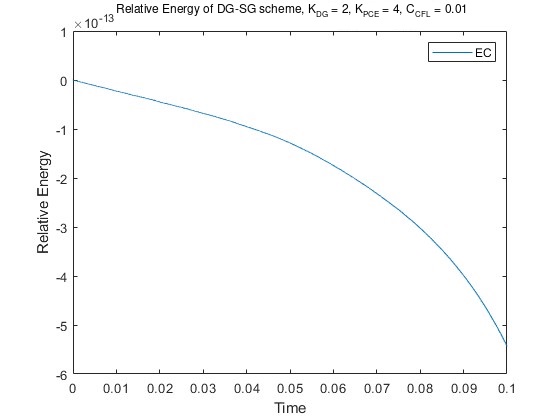}
        \captionsetup{font={scriptsize},justification=raggedright}
        \caption{EC, $C_{\rm CFL}=0.01$}
    \end{subfigure}
    \captionsetup{font={scriptsize},justification=raggedright}
    \caption{Relative energy evolution for the accuracy test in Section \ref{eg_accuracytest}. Left: comparison of entropy-conservative (EC) and entropy-stable (ES) fluxes with $C_{\rm CFL}=0.1$. Right: entropy-conservative flux with $C_{\rm CFL}=0.01$.}
    \label{fig1}
\end{figure}

\subsection{Well-balanced test}

 In this example, we test the well-balanced property of our proposed methods to ensure that the lake-at-rest steady state is exactly preserved. This is a stochastic modification of the well-balanced test in \cite{fu2022high}, with different choices of bottom topography. The stochastic smooth bottom is given by
 \begin{equation}\label{eg2_bottom1}
     B(x,\xi) = 5\exp \big( -0.4(x-5)^2\big) + 0.01\xi,
 \end{equation}
 and the stochastic discontinuous bottom is given by
 \begin{equation}\label{eg2_bottom2}
     B(x,\xi) = \left\{
\begin{aligned}
&4+0.01\xi,&  &4<x<8, \\
  &0,&  &\text{otherwise},
\end{aligned}
\right.
 \end{equation}
 and the steady state initial condition is given by
 \begin{align}
     w(x,0)& = 10,& q(x,0) &= 0.
 \end{align}

\begin{table}[htbp]
    \centering
    \begin{tabular}{llllll}
    \noalign{\smallskip}
\hline
\noalign{\smallskip}
\# of elements  &$L^2$ Error in $h_{\Delta x,s}$ & $L^2$ Error in $u_{\Delta x,s}$ & $L^2$ Error in $q_{\Delta x,s}$ \\
\noalign{\smallskip}

$100$  &  1.4666e-14 & 2.1661e-13    &  1.8548e-12  \\

$200 $  &2.2580e-14 &  4.4469e-13  &    3.8231e-12  \\
$400$  &   2.5704e-14  &  8.1617e-13  &      6.9626e-12 \\

\hline
\end{tabular}
\captionsetup{font={scriptsize},justification=raggedright}
  \caption{Well-balanced test with smooth stochastic bottom in \eqref{eg2_bottom1}. $T=0.5$. Third order DG with $K_{\rm DG}=2$ and $K_{\rm PCE} = 4$.}
  \vspace{-0.5cm}
\label{table2}
\end{table}

\begin{table}[htbp]
    \centering
    \begin{tabular}{llllll}
    \noalign{\smallskip}
\hline
\noalign{\smallskip}
\# of elements  &$L^2$ Error in $h_{\Delta x,s}$ & $L^2$ Error in $u_{\Delta x,s}$ & $L^2$ Error in $q_{\Delta x,s}$ \\
\noalign{\smallskip}

$100$  &  2.4975e-14  & 3.2160e-13    &  2.3598e-12  \\

$200 $  & 2.7255e-14 &  6.5619e-13  &    4.7504e-12\\
$400$  &    2.7366e-14  &  1.3076e-12  &      9.4330e-12 \\

\hline
\end{tabular}
\captionsetup{font={scriptsize},justification=raggedright}
  \caption{Well-balanced test with discontinuous stochastic bottom in \eqref{eg2_bottom2}. $T=0.5$. Third order DG with $K_{\rm DG}=2$ and $K_{\rm PCE} = 4$.}
\label{table3}
\end{table}

We solve the problem in the computational domain $[0,10]$ with wall boundary conditions \cite{fu2022high}, up to time $t=0.5$ on three uniform meshes with $100, 200$, and $400$ cells, and compute the $L^2(\Omega;L_{\rho}^2(\mathbb{R}^d))$ errors defined in \eqref{error_def} in Table \ref{table2} and \ref{table3}. The errors are measured against the initial steady state at the final time. All errors are at the level of round-off, confirming the well-balanced property of the proposed scheme.

\subsection{Another well-balanced test with near-dry region}\label{sec_eg5-3}

We consider another well-balanced test with a near-dry region. The test is modified from \cite{dai2021hyperbolicity}, with a stochastic water surface,
\begin{align}
    w(x,0,\xi) &= \left\{
\begin{aligned}
&1 + 0.001(\xi+1), & 0.1<x<0.2, \\
  &1, &\text{otherwise,}
\end{aligned}
\right.& \xi &\sim \mathcal{U}(-1,1),& u(x,0)&=0,
\end{align}
and a deterministic bottom topography
\begin{align}
B(x) &= \left\{
\begin{aligned}
&9.995(x - 0.3),  &0.3 \le x \le 0.4, \\
&0.9995 - 0.0025 \sin^2\!\big(25\pi (x - 0.4)\big), & 0.4 \le x \le 0.6, \\
&-9.995(x - 0.7),  &0.6 \le x \le 0.7, \\
&0, & \text{otherwise}.
\end{aligned}
\right.
\end{align}
This example is designed to assess the robustness of the proposed DG-SG scheme when the water height becomes small and the hyperbolicity-preserving criterion and velocity desingularization become important. The computational domain is $[-1,1]$, up to time $t=1$. A small uncertainty region was originally located at $x \in (0.1,0.2)$, where the water surface is slightly perturbed. The positivity-preserving criterion is activated due to the existence of a near-dry region, where the water height $h$ is relatively small. We compute the solution on a uniform mesh with mesh size $\Delta x = 1/100$, and plot the mean, standard deviation and quantile region, of water surface and discharge, respectively, in Figure \ref{fig2}. We also observe that the energy decays throughout the simulation.

\begin{figure}[htbp]
    \centering
    \includegraphics[width=1\linewidth]{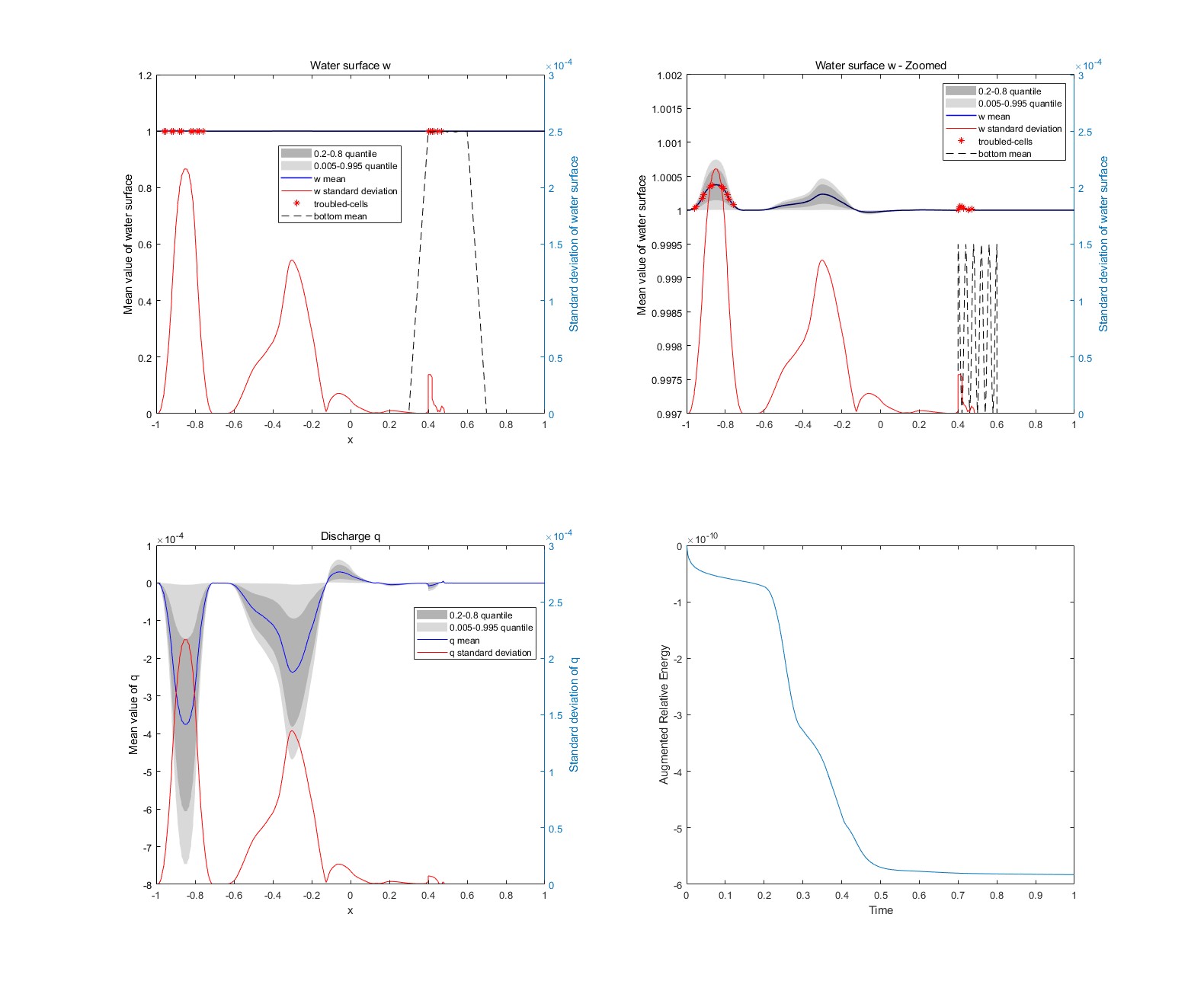}
    \captionsetup{font={scriptsize},justification=raggedright}
    \caption{Mean, standard deviation, quantile region, and augmented relative energy evolution for the near-dry perturbation test in Section \ref{sec_eg5-3}. $\Delta x = 1/100$. $T=1$. Third order DG with $K_{\rm DG}=2$ and $K_{\rm PCE} = 4$.}
    \label{fig2}
\end{figure}

The initial stochastic perturbation on the water surface is $10^{-3}$. Throughout the time evolution, the 
maximum water surface displacement from the steady state remains within its initial perturbation threshold, demonstrating the capability of our proposed scheme to accurately capture the lake-at-rest solution under small perturbations.

\subsection{Perturbation to lake at rest}\label{sec_eg5-4}
We consider an example proposed in \cite{dai2024energy,chertock2015well}, which is a perturbation to the lake-at-rest solution with a stochastic initial water surface
\begin{align}
    w(x,0,\xi) &= \left\{
\begin{aligned}
&1 + 0.001(\xi+1), & |x|\le 0.05, \\
  &1, &\text{otherwise,}
\end{aligned}
\right.& u(x,0)&=0,
\end{align}
and with the deterministic bottom topography
\begin{equation}\label{eq:lar-B1}
    B(x) = \left\{
\begin{aligned}
& 0.25(\cos(5\pi(x+0.35))+1), & &-0.55<x<-0.15, \\
  & 0.125(\cos(10\pi(x-0.35))+1), & &0.25<x<0.45,\\
  & 0, & &\text{otherwise.}
\end{aligned}
\right.
\end{equation}

The computational domain is $[-1,1]$, and $\xi \sim \mathcal{U}(-1,1)$. In this example, we choose a relatively large tolerance $\mathrm{tol}=0.05$ in the troubled-cell indicator to reduce the number of detected troubled cells and thereby improve resolution. We compute the DG solution up to time $t=0.8$ on meshes with $N_x=200$ and $400$ cells. A third-order DG method (polynomial degree $K_{\rm DG}=2$) is used for the spatial discretization, together with $K_{\rm PCE}=4$ terms in the PCE for the stochastic Galerkin approximation.

As shown in Figure \ref{fig3}, the proposed DG-SG scheme accurately captures both left- and right-going waves of small amplitude, including the propagation of uncertainty. This behavior is consistent with the well-balanced structure of the DG-SG scheme. As expected, the finer mesh yields improved solution resolution. Moreover, the DG results provide higher resolution than the finite volume (FV) results reported in \cite{dai2024energy} on comparable meshes. From the
energy plot in Figure \ref{fig4}, the energy decays as the theoretical result. In addition, the solution with the finer mesh demonstrates less numerical dissipation.

\begin{figure}[htbp]
    \centering
    \includegraphics[width=1\linewidth]{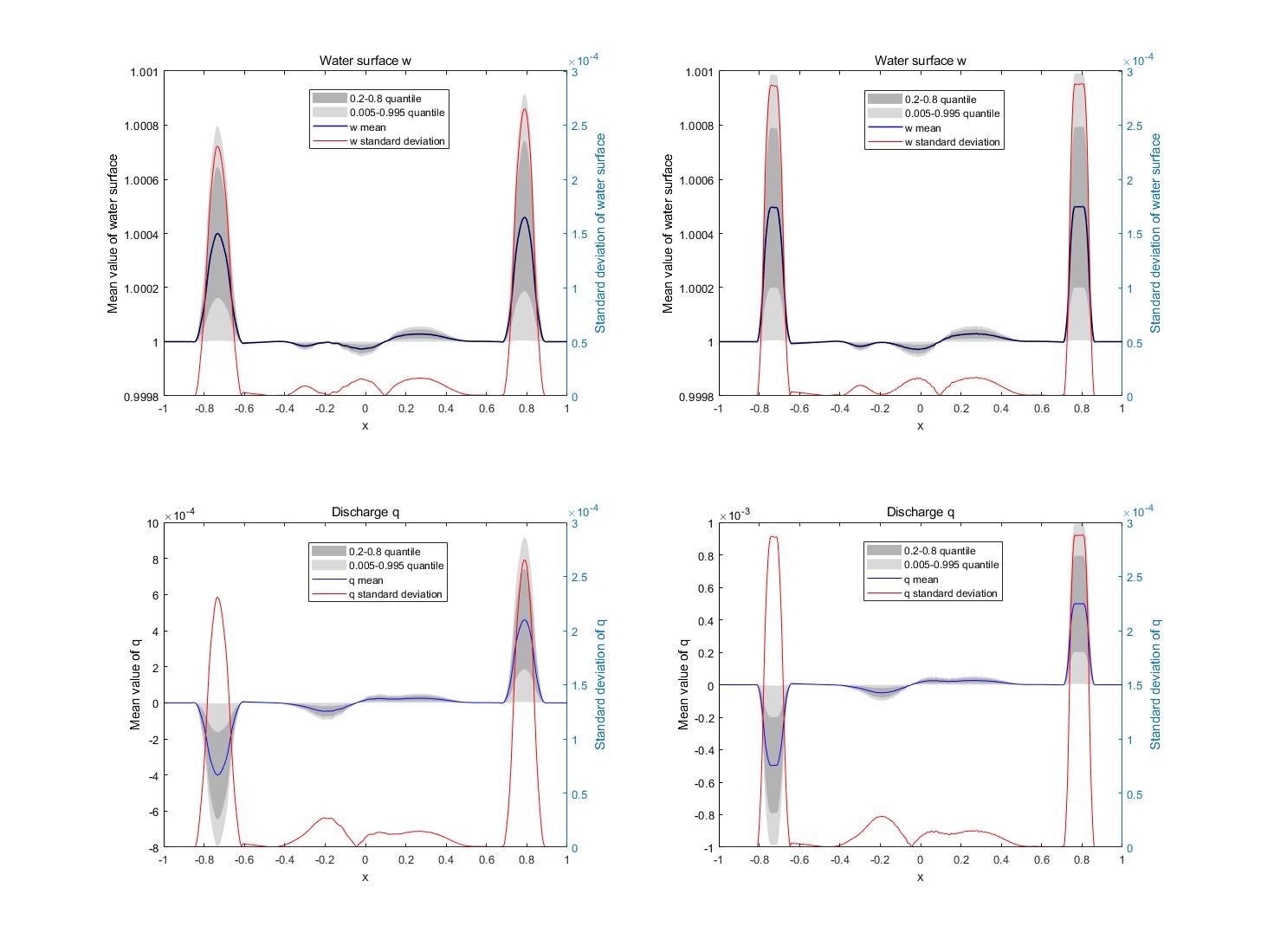}
    \captionsetup{font={scriptsize},justification=raggedright}
    \caption{Mean, standard deviation, quantile region of the water surface and discharge for the perturbation-to-lake-at-rest problem in Section \ref{sec_eg5-4} at $T=0.8$. Left: $N_x=200$. Right: $N_x=400$. Third-order DG with $K_{\rm DG}=2$ and $K_{\rm PCE}=4$.}
    \label{fig3}
\end{figure}

\begin{figure}[htbp]
    \centering
    \includegraphics[width=1\linewidth]{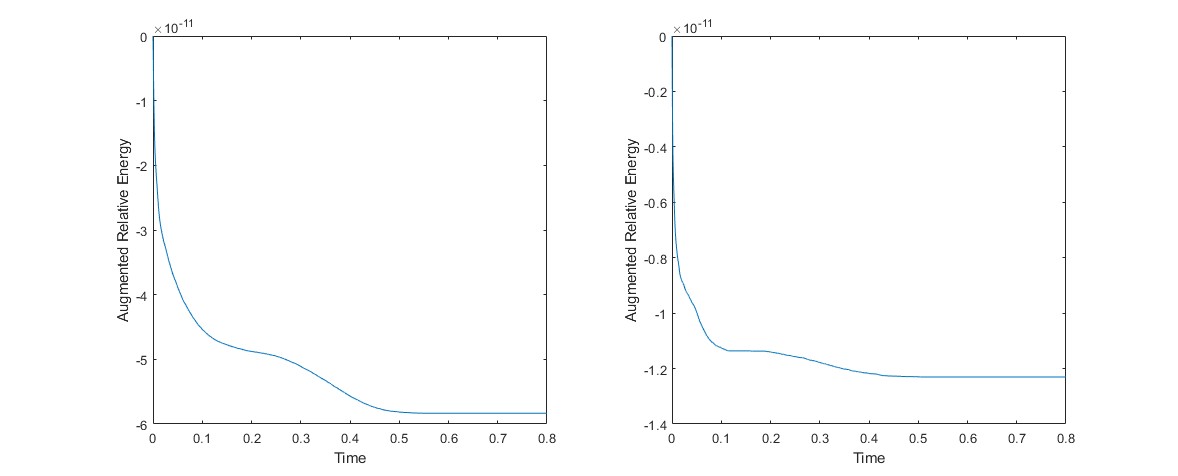}
    \captionsetup{font={scriptsize},justification=raggedright}
   \caption{Augmented relative energy evolution for the perturbation-to-lake-at-rest problem in Section \ref{sec_eg5-4}. Left: $N_x=200$. Right: $N_x=400$. Third-order DG with $K_{\rm DG}=2$ and $K_{\rm PCE}=4$.}
    \label{fig4}
\end{figure}

In addition, we illustrate the accuracy of the DG-SG schemes on this problem with a discontinuous initial water surface in Table \ref{table4} and Table \ref{table5} at different evolution time. From the tables, we observe that our proposed DG-SG scheme performs well on problems with nonsmooth solutions. Note that the order of convergence is limited by the solution regularity and hence cannot achieve the optimal convergence rate, as expected. In this discontinuous problem, we also introduce the $L^1(\Omega;L_{\rho}^2(\mathbb{R}^d))$ error used in \cite{dai2024energy} as \eqref{error_def_l1}. The results are reported in Table \ref{table4} and Table \ref{table5}, which are consistent with the FVM result in \cite{dai2024energy}. 
\begin{equation}\label{error_def_l1}
\begin{split}
    \text{Error}(h_{\Delta x, s}) &= \| h_{\Delta x, s}(x,t,\xi)-h_{ref}(x,t,\xi)\|_{L^1(\Omega;L_{\rho}^2(\mathbb{R}^d))},\\
    \|h(x,t,\xi )\|_{L^1(\Omega;L_{\rho}^2(\mathbb{R}^d))} & \coloneqq \int_{\Omega}\|h(x,t,\xi)\|_{L_{\rho}^2}dx
\end{split}
\end{equation}

\begin{table}[htbp]
    \centering
    \begin{tabular}{lllllll}
    \noalign{\smallskip}
\hline
\noalign{\smallskip}
Scheme & Grid size  & Error $L^1 L^2_{\rho}$ & Order & Error $L^2 L^2_{\rho}$ & Order \\

\hline
\multirow{4}{5em}{$K_{\rm DG}=2$\\ $K_{\rm PCE}$=2}& $50$
 &   4.5994e-04 &  &  6.2557e-04  &\\

&$100$  &  1.5884e-04 & 1.5339  & 2.0524e-04&   1.6079\\

&$200 $  & 4.7583e-05 & 1.7390  & 7.5775e-05 & 1.4375\\
&$400$  & 1.3179e-05   &   1.8523 & 3.0206e-05   & 1.3269\\
\hline
\end{tabular}
\captionsetup{font={scriptsize},justification=raggedright}
  \caption{Accuracy test for the perturbation-to-lake-at-rest problem in Section \ref{sec_eg5-4}. Order of convergence for $K_{\rm DG}=2$, $K_{\rm PCE}=2$, $T=0.8$. Reference solution computed with $800$ elements. $C_{\rm CFL} = 0.1$.}
\label{table4}
\end{table}

\begin{table}[htbp]
    \centering
    \begin{tabular}{lllllll}
    \noalign{\smallskip}
\hline
\noalign{\smallskip}
Scheme & Grid size  & Error $L^1 L^2_{\rho}$ & Order & Error $L^2 L^2_{\rho}$ & Order \\

\hline
\multirow{4}{5em}{$K_{\rm DG}=2$\\ $K_{\rm PCE}$=2}& $50$
 &   4.6058e-04 &  &  6.2550e-04  &\\

&$100$  &  1.6004e-04 & 1.5250  & 2.0582e-04&   1.6037\\

&$200 $  & 4.7728e-05 & 1.7455  & 7.5782e-05 & 1.4414\\
&$400$  & 1.3215e-05   &   1.8526 & 3.0205e-05   & 1.3270\\
\hline
\end{tabular}
\captionsetup{font={scriptsize},justification=raggedright}
  \caption{Accuracy test for the perturbation-to-lake-at-rest problem in Section \ref{sec_eg5-4}. Order of convergence for $K_{\rm DG}=2$, $K_{\rm PCE}=2$, $T=0.85$. Reference solution computed with $800$ elements. $C_{\rm CFL} = 0.1$.}
\label{table5}
\end{table}

\subsection{Perturbation to lake at rest with a two-dimensional random variable}\label{sec_eg5-5}
As a final test, we consider a two-dimensional stochastic space variant of the example used for the perturbation to lake at rest in Section \ref{sec_eg5-4}. We consider the shallow water system with a stochastic water surface
\begin{equation}
w(x,0,\xi) =
\begin{cases}
1 + 0.001(\xi_1 + 1), & |x| \le 0.05, \\
1, & \text{otherwise},
\end{cases}
\qquad
q(x,0) = 0,
\end{equation}
and with a stochastic bottom topography
\begin{equation}
B(x,\xi) =
\begin{cases}
0.25\bigl(\cos(5\pi(x + 0.35)) + 1\bigr) + 0.12 e^{\xi_2}, 
& -0.55 < x < -0.15, \\
0.125\bigl(\cos(10\pi(x - 0.35)) + 1\bigr) + 0.1(1 + \xi_1), 
& 0.25 < x < 0.45, \\
0, & \text{otherwise}.
\end{cases}
\end{equation}
where $\xi_1$ and $ \xi_2$ are independent, beta-distributed variables each with parameters $(\alpha,\beta)=(1,3)$. We use $K_1=K_2=4$ PCE modes for the random variables $\xi_1$ and $\xi_2$, resulting in a tensor-product polynomial space of dimension $K_{\rm PCE}=K_1K_2=16$. We compute the DG solution up to time $t=0.8$ on meshes with $N_x=200$ and $400$ cells, using a third-order DG method (polynomial degree $K_{\rm DG}=2$) for spatial discretization.
\begin{figure}[htbp]
    \centering
    \includegraphics[width=1\linewidth]{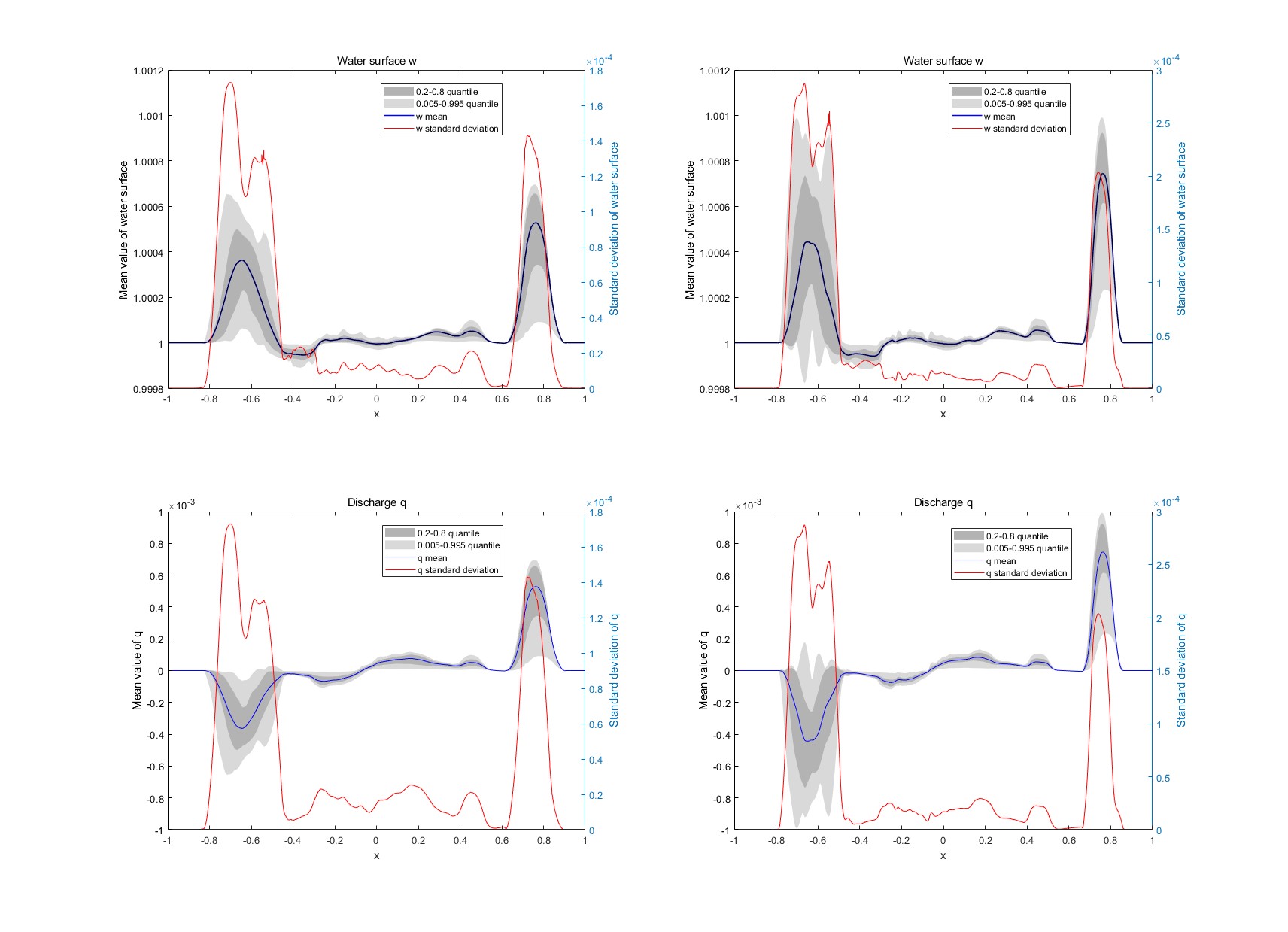}
    \captionsetup{font={scriptsize},justification=raggedright}
    \caption{Mean, standard deviation, quantile region of the water surface and discharge plot of Section \ref{sec_eg5-5}. Left: solution with $N_x=200$ elements. Right: solution with $N_x=400$ elements. Third order DG with $K_{\rm DG}=2$ and $K_{\rm PCE} = K_1K_2 = 16$.}
    \label{fig5}
\end{figure}

\begin{figure}[htbp]
    \centering
    \includegraphics[width=1\linewidth]{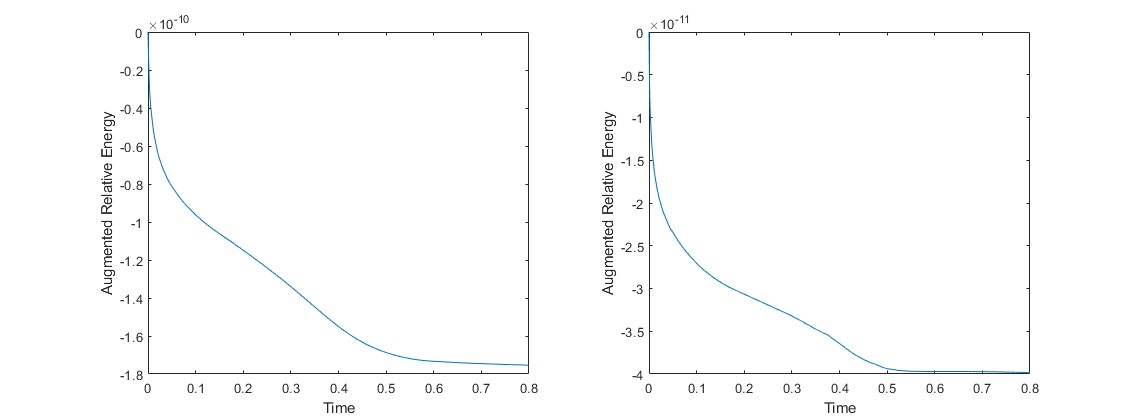}
    \captionsetup{font={scriptsize},justification=raggedright}
    \caption{Augmented relative energy evolution for the two-dimensional stochastic-input test in Section \ref{sec_eg5-5}. Left: $N_x=200$. Right: $N_x=400$. Third-order DG with $K_{\rm DG}=2$ and $K_{\rm PCE}=K_1K_2=16$.}
    \label{fig6}
\end{figure}

From the solution plot in Figure \ref{fig5}, the DG solutions exhibit more complex wave structures than those in Section \ref{sec_eg5-4}, due to the more intricate uncertainties embedded in the bottom topography. These results demonstrate that the proposed DG-SG scheme remains effective for multi-dimensional stochastic inputs and can resolve the interaction between uncertainty in the initial water surface and uncertainty in the bottom topography. In addition, the energy plots in Figure \ref{fig6} are consistent with those in Section \ref{sec_eg5-4}, supporting the entropy stability of the method in the multivariate stochastic setting.

\section{Conclusion}\label{Sec6}
In this work, we developed a semi-discrete discontinuous Galerkin–stochastic Galerkin (DG-SG) scheme for the hyperbolicity-preserving skew-symmetric SG SWE system. Using the skew-symmetric structure, we constructed well-balanced, energy conservative and energy stable numerical fluxes for the proposed numerical schemes, ensuring the discrete entropy admissibility criteria for selecting the desired physical solution among non-unique weak solutions.  We presented several numerical experiments demonstrating the efficiency and robustness of our schemes. Currently, our schemes are constructed on one-physical dimension; one future direction is to extend the present framework to two-physical dimension with unstructured meshes. Another possibility for future research is to extend the developed framework to models associated with dry/wet interfaces.

\section*{Acknowledgement}
The work of Yekaterina Epshteyn, Akil Narayan and Yinqian Yu was partially supported by NSF DMS-2207207. AN was also partially supported
by NSF DMS-1848508.

\clearpage
\footnotesize{\bibliographystyle{siam}} 
\bibliography{ref} 

\end{document}